\documentclass[a4paper,11pt]{amsart}
\usepackage[dvipdfmx]{graphicx}

\usepackage{comment}

\usepackage{etoolbox}
\ifdef{\crop}{%
\usepackage[includeheadfoot,twoside=False,paperwidth=448pt,paperheight=587pt,rmargin=15pt,lmargin=15pt,tmargin=15pt,bmargin=15pt]{geometry}%
}{%
\setlength{\topmargin}{20mm}
\addtolength{\topmargin}{-1in}
\setlength{\oddsidemargin}{27mm}
\addtolength{\oddsidemargin}{-1in}
\setlength{\evensidemargin}{27mm}
\addtolength{\evensidemargin}{-1in}
\setlength{\textwidth}{156mm}
\setlength{\textheight}{242mm}
}%

\usepackage{color}
\usepackage[active]{srcltx}
\usepackage{amsmath,amsthm,amsxtra}
\usepackage{amssymb}

\usepackage[mathscr]{eucal}

\usepackage{bm}
\usepackage[all]{xy}

\usepackage{aliascnt}

\usepackage{tabularx}
\newcolumntype{C}{>{\centering\arraybackslash}X} 

\theoremstyle{plain}
\newtheorem{thm}{Theorem}[section]
\newtheorem*{thm*}{Theorem}

\newaliascnt{prop}{thm}
\newaliascnt{cor}{thm}
\newaliascnt{lem}{thm}
\newaliascnt{claim}{thm}
\newaliascnt{defn}{thm}
\newaliascnt{ques}{thm}
\newaliascnt{conj}{thm}
\newaliascnt{fact}{thm}
\newaliascnt{rem}{thm}
\newaliascnt{ex}{thm}
\newtheorem{prop}[prop]{Proposition}
\newtheorem{cor}[cor]{Corollary}
\newtheorem{lem}[lem]{Lemma}
\newtheorem{claim}[claim]{Claim}
\newtheorem*{prop*}{Proposition}
\newtheorem*{cor*}{Corollary}
\newtheorem*{lem*}{Lemma}
\newtheorem*{claim*}{Claim}
\theoremstyle{definition}

\newtheorem{conj}[conj]{Conjecture}
\newtheorem*{defn*}{Definition}
\newtheorem*{ques*}{Question}
\newtheorem*{conj*}{Conjecture}

\newtheorem*{prob*}{Problem}

\newtheorem{rem}[rem]{Remark}
\newtheorem{ex}[ex]{Example}
\newtheorem*{fact*}{Fact}
\newtheorem*{rem*}{Remark}
\newtheorem*{ex*}{Example}
\aliascntresetthe{prop}
\aliascntresetthe{cor}
\aliascntresetthe{lem}
\aliascntresetthe{claim}
\aliascntresetthe{defn}
\aliascntresetthe{ques}
\aliascntresetthe{conj}
\aliascntresetthe{fact}
\aliascntresetthe{rem}
\aliascntresetthe{ex}

\usepackage[pointedenum]{paralist}
\usepackage{varioref}
\labelformat{equation}{\textnormal{(#1)}}
\labelformat{enumi}{\textnormal{(#1)}}

\usepackage[a4paper]{hyperref}

\def\textsectionN~{\textsection{}}

\renewcommand\phi{\varphi}
\renewcommand\epsilon{\varepsilon}
\renewcommand\leq{\leqslant}
\renewcommand\geq{\geqslant}

\makeatletter
\newcommand{\set}{  \@ifstar{\@setstar}{\@set}}\newcommand{\@setstar}[2]{\{\, #1 \mid #2 \,\}}
\newcommand{\@set}[1]{\{ #1 \}}
\makeatother

\newcommand{\trans}[1][1]{\raisebox{#1ex}{\scriptsize\kern0.1em$t$\kern-0.1em}}

\DeclareMathOperator{\codim}{codim}

\DeclareMathOperator{\Proj}{Proj}

\DeclareMathOperator{\Pic}{Pic}

\DeclareMathOperator{\Sym}{Sym}

\DeclareMathOperator{\Ext}{Ext}
\DeclareMathOperator{\Char}{char}

\def\Z{\mathbb{Z}}
\def\Q{\mathbb{Q}}
\def\R{\mathbb{R}}
\def\C{\mathbb{C}}

\def\r+{\mathbb{R}_{\geq 0}}

\def\ep{\varepsilon}

\def\r+{{\R}_{\geq 0}}
\def\q+{{\Q}_{\geq 0}}
\def\P{\mathbb{P}}

\def\*c{\C^{\times}}

\def\dX{\widehat{X}}

\def\C{\mathbb {C}}

\def\K{\mathbb {K}}

\def\Q{\mathbb {Q}}
\def\R{\mathbb {R}}

\def\Z{\mathbb {Z}}

\newcommand{\cale}{\mathcal {E}}
\newcommand{\calf}{\mathcal {F}}
\newcommand{\calg}{\mathcal {G}}

\newcommand{\cali}{\mathcal {I}}

\newcommand{\calo}{\mathcal {O}}
\newcommand{\calp}{\mathcal {P}}

\makeatletter
  
  \@addtoreset{equation}{section}
\makeatother

\title[M-regularity and linear systems on abelian varieties]{M-regularity of $\mathbb{Q}$-twisted sheaves and its application to linear systems on abelian varieties}

\author[A.~Ito]{Atsushi~Ito}
\address{Graduate School of Mathematics,
Nagoya University,
Nagoya, Japan}
\email{atsushi.ito@math.nagoya-u.ac.jp}

\subjclass[2010]{14C20,14K99}
\keywords{Abelian variety, M-regularity, Linear system}

\begin{document}

\maketitle

\begin{abstract}
G.~Pareschi and M.~Popa give criterions for global generations and surjectivity of multiplication maps of global sections of coherent sheaves on abelian varieties
in the theory of M-regularity.
In this paper,
we generalize some of their criterions via the M-regularity of $\mathbb{Q}$-twisted sheaves.
As an application,
we show that the M-regularity of a suitable $\Q$-twisted sheaf
implies property $(N_p)$ and jet-ampleness for ample line bundles on abelian varieties.
\end{abstract}

\section{Introduction}\label{sec_intro}

Throughout this paper we work over an algebraically closed field $\K$.
In \cite{MR1949161}, 
G.~Pareschi and M.~Popa introduce the notion of \emph{M-regularity} 
as follows:

For a coherent sheaf $\calf$ on an abelian variety $X$ defined over $\K$,
set 
\[
V^i(\calf) =  \{ \alpha \in \widehat{X} \, | \,  h^i(X,\calf \otimes P_{\alpha}) >0 \},
\]
where $\dX$ is the dual abelian variety of $X$ and $P_{\alpha}$ is the algebraically trivial line bundle on $X$ corresponding to $\alpha \in \dX$.
Then $\calf$ is said to be \emph{GV} if 
$\codim_{\widehat{X}} V^i(\calf) \geq i$ for any $i >0$.
It is said to be \emph{M-regular} if $\codim_{\widehat{X}} V^i(\calf) > i$ for any $i >0$.
It is said to be \emph{IT(0)} if $V^i(\calf) =\emptyset$ for any $i >0$.

In \cite{MR1949161}, \cite{MR2008719}, \cite{MR2807853}, etc.,
Pareschi and Popa develop the theory of M-regularity and give many applications.
M-regularity is useful since it implies suitable globally generation and surjectivity of multiplication maps of global sections,
as in the case of Castelnuovo-Mumford regularity on projective spaces.
Among others, they prove the following results:

\begin{thm}[{\cite[Theorem 6.3]{MR1949161}, \cite[Theorem 7.34]{MR2807853}}]\label{thm_Pareschi-Popa_criterion}
Let $A$ be an ample line bundle on an abelian variety $X$
and $\cale, \calf  $ be coherent sheaves on $X$.
\begin{enumerate}
\item If $\calf \otimes A^{-1}$ is M-regular, then $\calf$ is globally generated.
\item If $\calf \otimes A^{-2}$ is M-regular, then the natural map 
$H^0(A^n) \otimes H^0(\calf) \rightarrow H^0(A^n \otimes \calf) $ is surjective for any $n \geq 2$.
\item If $\cale, \calf$ are locally free and $\cale \otimes A^{-2}, \calf \otimes A^{-2}$ are M-regular,
then the natural map $H^0(\cale) \otimes H^0( \calf ) \rightarrow H^0(\cale \otimes \calf)$ is surjective.
\end{enumerate}
\end{thm}

Recently 
Z.~Jiang and G.~Pareschi \cite{MR4157109} extend 
the notions such as GV, M-regular, IT(0) to a \emph{$\Q$-twisted sheaf} $\calf \langle xl \rangle $,
where 
$x \in \Q$ and $l \in \Pic X /\Pic^0 X$ is the class of an ample line bundle $L$ (see \autoref{subsec_notion_GV} for the definition of $\Q$-twisted sheaves).
In \cite{MR4157109},
the authors also define an invariant $0 < \beta(l) \leq 1$, which is characterized as
\[
\beta(l) < x \quad  \Longleftrightarrow  \quad \cali_o \langle x l \rangle \text{ is IT(0)}
\]
for $x \in \Q$, 
where $ \cali_o \subset \calo_X$ is the maximal ideal corresponding to the origin $o \in X$.

The first purpose of this paper is to generalize \autoref{thm_Pareschi-Popa_criterion} to the $\Q$-twisted setting as follows:

\begin{thm}[{Propositions \ref{lemma_globally_generated}, \ref{prop_surjectiveity_E,F}}]\label{thm_criterion}
Let $L$ be an ample line bundle,
$\cale$ be a locally free sheaf, and $\calf  $ be a coherent sheaf on an abelian variety $X$.
Take 
$x \in \Q$ such that $x \geq \beta(l)$.
\begin{enumerate}
\item 
%
If $ \calf \langle -x l \rangle$ is M-regular,
then $\calf$ is globally generated.
\item If $x <1$ and $\calf \langle -\frac{x}{1-x} l\rangle$ is M-regular,  
then the natural map $H^0(L) \otimes H^0(\calf) \rightarrow H^0(L \otimes \calf)$ is surjective.
\item 
If there exist rational numbers $s,t >0$ such that 
$\cale  \langle -s l \rangle, \calf  \langle -t l \rangle$ are M-regular and $st/(s+t) \geq \beta(l) $,
then the natural map $H^0(\cale) \otimes H^0( \calf ) \rightarrow H^0(\cale \otimes \calf)$ is surjective.
\end{enumerate}
\end{thm}

We note that $\calf \langle ml \rangle$ is M-regular if and only if so is $\calf \otimes L^{m}$ for $m \in \Z$.
Hence \autoref{thm_Pareschi-Popa_criterion} (1) is nothing but the case $A=L, x=1$ of \autoref{thm_criterion} (1).
Similarly, \autoref{thm_Pareschi-Popa_criterion} (2) follows from the case $L=A^n, x= 1/n$ of \autoref{thm_criterion} (2),
and \autoref{thm_Pareschi-Popa_criterion} (3) follows from the case $A=L, s=t=2$ of \autoref{thm_criterion} (3).

\vspace{2mm}
The second purpose of this paper is to study linear systems on abelian varieties by using \autoref{thm_criterion}
as in \cite{MR2008719}.
For an ample line bundle $A$ on an abelian variety $X$,
\begin{itemize}
\item $A^n$ is basepoint free if $n \geq 2$ (Lefschetz Theorem),
\item $A^n$ is projectively normal if $n \geq 3$ (\cite{MR0282975}, \cite{MR480543}),
\item the homogeneous ideal of $X$ embedded by $|A^n|$ is generated by quadrics if $n \geq 4$ (\cite{MR204427}, \cite{MR980300}).
\end{itemize}

As a generalization of these results,
R.~Lazarsfeld conjectures that $A^n$ satisfies \emph{property} $(N_p)$ if $n \geq p+3$ for $p \geq 0$. 
See \autoref{subsec_N_p} for the definition of $(N_p)$.
We just note here that 
($N_0$) holds for an ample line bundle $L$ if and only if $L$ defines a projectively normal embedding,
and ($N_1$) holds if and only if ($N_0$) holds and the homogeneous ideal of the embedding is generated by quadrics.

Lazarsfeld's conjecture is proved by Pareschi \cite{MR1758758} in 
$\mathrm{char} (\K)=0$.
In \cite{MR2008719},
Pareschi and Popa strengthen Lazarsfeld's conjecture
when $A$ has no base divisor, that is, when $\codim_X \mathrm{Bs} |A| \geq 2$, by using \autoref{thm_Pareschi-Popa_criterion}:

\begin{thm}[{\cite[Theorem 6.2]{MR2008719}}]\label{thm_Pareschi-PopaII}
Let $p \geq 1$ 
be an integer such that $\mathrm{char} (\K)$ does not divide $p+1$ and $p+2$. 
Let $A $ be an ample line bundle with no base divisor on an abelian variety $X$.
\begin{enumerate}
\item If $n \geq p+2$, then $A^n$ satisfies $(N_p)$.
\item More generally, if  $n \geq (p+r+2)/(r+1)$, 
 then $A^n$ satisfies $(N_p^r)$ for $r \geq 0$.
\end{enumerate}
\end{thm}

Property $(N_p^r)$ for $p,r \geq 0$ is introduced in \cite{MR1758758},
where it is proved that $A^n$ satisfies $(N_p^r)$ if $  n \geq (p+r+3)/(r+1)$ 
without the assumption on the base divisors of $A$.
See \autoref{subsec_N_p} for the definition of $(N_p^r)$.
We just note here that $(N_p^0)$ is equivalent to $(N_p)$ and
$(N_p^r) $ is a property of ``being off'' by $r$ from $(N_p)$.

On the other hand,
the following theorem by Jiang-Pareschi and F.~Caucci generalizes Lazarsfeld's conjecture  to 
ample line bundles which is  not necessarily multiples of another line bundles:

\begin{thm}[{\cite[Section 8]{MR4157109},\cite[Theorem 1.1]{MR4114062}}]\label{thm_JP_Caucci}
Let $L$ be an ample line bundle on an abelian variety $X$ and $p \geq 0$. 
Let $\cali_o \subset \calo_X$ be the maximal ideal sheaf corresponding to the origin  $o \in X$.
Then 
\begin{enumerate}
\item $\cali_o \langle xl \rangle $ is IT(0) for $ 1 < x \in \Q$, and
$\cali_o \langle l \rangle$ is IT(0) if and only if $L$ is basepoint free.
\item If $ \cali_o \langle \frac{1}{p+2} l \rangle$ is IT(0),
then $L$ satisfies $(N_p)$.
\end{enumerate}
\end{thm}

\autoref{thm_JP_Caucci} gives a quick and characteristic-free proof of Lazarsfeld's conjecture as follows:
If $L=A^n$ for some $A$ and $n \geq 1$, then
$ \cali_o \langle \frac{1}{p+2} l \rangle  = \cali_o \langle \frac{n}{p+2} a \rangle $ is IT(0) if $n/(p+2) >1$ by \autoref{thm_JP_Caucci} (1).
Hence 
$A^n$ satisfies $(N_p)$ if $n \geq p+3$ by \autoref{thm_JP_Caucci} (2).
Furthermore, the proof of 
\autoref{thm_JP_Caucci} shows that $L$ satisfies $(N_p^r)$ for $r \geq 0$ if $\cali_o \langle \frac{r+1}{p+r+2} l \rangle $ is IT(0).

Applying \autoref{thm_criterion},
we obtain the following theorem,
which contains \autoref{thm_Pareschi-PopaII} and
the case $p \geq 1$ of \autoref{thm_JP_Caucci} (2):

\begin{thm}\label{thm_M-regular_is_enough}
Let $L$ be an ample line bundle on an abelian variety $X$ and  $p \geq 1$.
\begin{enumerate}
\item 
If $\cali_o \langle \frac{1}{p+2} l \rangle $ is M-regular, then $L$ satisfies $(N_p)$.
\item 
More generally, if $\cali_o \langle \frac{r+1}{p+r+2} l \rangle $ is M-regular, then $L$ satisfies $(N_p^r)$ for $r \geq 0$.
\end{enumerate}
\end{thm}

By \cite[Remark 3.6]{MR2008719},
$\cali_o \langle l \rangle$ is M-regular if and only if $L$ has no base divisor.
Hence \autoref{thm_M-regular_is_enough} gives a characteristic-free proof of \autoref{thm_Pareschi-PopaII}:
If $A$ has no base divisor and
$L=A^n$,
then 
$\cali_o \langle x a \rangle$ is M-regular for $x \geq 1$ and hence
$\cali_o \langle \frac{1}{p+2} l \rangle = \cali_o \langle \frac{n}{p+2} a \rangle$ is M-regular if $n/(p+2)  \geq 1$.
Thus 
\autoref{thm_Pareschi-PopaII} (1) holds in any characteristic.
Similarly, \autoref{thm_M-regular_is_enough} (2) recovers \autoref{thm_Pareschi-PopaII} (2) in any characteristic.



\vspace{2mm}

Other than $(N_p)$,
we can study \emph{jet-ampleness} via M-regularity as in \cite{MR2008719}.
Recall that a line bundle $L$ is called $k$-\emph{jet ample} for $k \geq 0$
if the restriction map 
\[
H^0 ( L ) \rightarrow H^0 ( L \otimes \calo_X/ \cali_{p_1}\cali_{p_2} \cdots \cali_{p_{k+1}}) 
\]
is surjective for any (not necessarily distinct) $k+1$ points $p_1,\dots,p_{k+1} \in X$.
In particular,
$0$-jet ampleness is equivalent to basepoint freeness and $1$-jet ampleness is equivalent to very ampleness.

For an ample line bundle $A$ on an abelian variety,
\cite{MR1439201} proves that $A^{n}$ is $k$-jet ample if $n \geq k+2$,
and the same holds if $A$ has no base divisor and $n \geq k+1 \geq 2$.
In \cite[Theorem 3.8]{MR2008719},
the authors generalize this result using the theory of M-regularity.
On the other hand, Caucci shows that $L $ is $k$-jet ample if $ \cali_o \langle \frac{1}{k+1} l \rangle$ is IT(0) \cite[Theorem D]{CaucciThesis},
which generalizes the first statement of the above result in \cite{MR1439201}. 
The following theorem generalizes the result in \cite{MR1439201} and improves \cite[Theorem D]{CaucciThesis}.
See \autoref{prop_jet_ampleness2} for a generalization of \cite[Theorem 3.8]{MR2008719}.

\begin{thm}\label{thm_jet-ample}
Let $L$ be an ample line bundle on an abelian variety $X$ and $k \geq 1$ be an integer.
If $ \cali_o \langle \frac{1}{k+1} l \rangle$ is M-regular,
then $L$ is $k$-jet ample.
\end{thm}

\begin{rem}\label{rem_p=0}
We note that the statement of \autoref{thm_Pareschi-PopaII} (1) does not hold for $p=0$ as we will see \autoref{ex_not_proj_normal}.
Hence \autoref{thm_M-regular_is_enough} (1) also fails for $p=0$,
that is, 
the M-regularity of $\cali_{o} \langle \frac12 l \rangle$ does not imply the projective normality of $L$ in general.
However,
the M-regularity of $\cali_{o} \langle \frac12 l \rangle$ implies the very ampleness of $L$ by \autoref{thm_jet-ample}.
See \autoref{sec_proj_normal} for the relations between these notions.

\autoref{thm_jet-ample} fails for $k =0$ as well
since the M-regularity of $ \cali_o \langle  l \rangle$, which is equivalent to say that $L$ has no base divisor,
is strictly weaker than the $0$-jet ampleness $=$ basepoint freeness. 
\end{rem}

This paper is organized as follows. 
In \autoref{sec_preliminary}, we recall some notation.
In \autoref{sec_global_generation},
we show 
\autoref{thm_criterion} (1). 
In \autoref{sec_M_L},
we show \autoref{thm_criterion} (2), (3).
In \autoref{sec_jet-ample},
we prove \autoref{thm_jet-ample}.
We also give an alternative proof of a characterization of polarized abelian varieties whose Seshadri constants are one by M.~Nakamaye \cite{MR1393263}.
In \autoref{sec_proof of theorem},
we prove \autoref{thm_M-regular_is_enough}.
In \autoref{sec_proj_normal},
we study projective normality.

\subsection*{Acknowledgments}
The author would like to thank Dr.\ Federico Caucci for useful communication and valuable comments.
The author was supported by JSPS KAKENHI Grant Number 17K14162.

\section{Preliminaries}\label{sec_preliminary}

Throughout this paper,
$X$ is an abelian variety of dimension $g$.
We denote the origin of $X$ by $o_X$ or $o \in X$.
Let $L$ be an ample line bundle on $X$ and $l \in \mathrm{NS}(X)=\Pic (X) / \Pic^0 (X)$ be the class of $L$ in the Neron-Severi group.
Then we have an isogeny
\[
\phi_l : X \rightarrow \dX:= \Pic^0 (X),  \quad p \mapsto t_p^* L \otimes L^{-1}
\]
which depends only on the class $l$,
where $t_p : X \rightarrow X$ is the translation by $p$.

For $b \in \Z$, we denote the multiplication-by-$b$ isogeny by
\[
\mu_b=\mu^X_b : X \rightarrow X , \quad p \mapsto bp.
\]
It holds that $\mu_b^* l=b^2 l$.
Since $\phi_l$ is a group homomorphism,
we have $\phi_{bl} = \hat{\mu}_b \circ \phi_l =\phi_l \circ \mu_b $,
where $ \hat{\mu}_{b} $ is the multiplication-by-$b$ isogeny on $\dX$.

For two line bundles $L,L'$,
$L \equiv L'$ means that $L,L'$ are algebraically equivalent.

\subsection{Properties $(N_p)$ and $(N_p^r)$}\label{subsec_N_p}

Let $L$ be a basepoint free ample line bundle $L$ on a projective variety $Y$
and set 
$S_L:=\Sym H^0(Y,L)$.
Take a minimal free resolution of $R_L:=\bigoplus_{n \geq 0} H^0(Y, L^{n})$ as an $S_L$-module
\begin{align}\label{eq_free_resolution}
\cdots \rightarrow E_p \rightarrow \dots \rightarrow E_1 \rightarrow E_0 \rightarrow R_L \rightarrow 0 \quad \text{with} \quad E_i \simeq \bigoplus_{j} S_L(-a_{ij}).
\end{align}
Then 
$L$ is said to \emph{satisfy property} $(N_p) $ if $E_0=S_L$ and $a_{ij} = i+1$ for any $1 \leq i \leq p$ and any $j$.
For example, 
($N_0$) holds for $L$ if and only if $L$ defines a projectively normal embedding,
and ($N_1$) holds if and only if ($N_0$) holds and the homogeneous ideal of the embedding is generated by quadrics.

More generally,
$L$ is said to \emph{satisfy property} $(N_0^r) $ if $ a_{0j} \leq 1+r$ for any $j$ in \ref{eq_free_resolution}.
Inductively,
$L$ is said to \emph{satisfy property} $(N_p^r) $ if $L$ satisfies $(N_{p-1}^r)  $ and $ a_{pj} \leq p+1+r$ for any $j$.
For example, $(N_p^0)$ is equivalent to $(N_p)$
and $L$ satisfies $(N_0^r)$ if and only if $H^0(L) \otimes H^0(L^n) \rightarrow H^0(L^{n+1}) $ is surjective for $n \geq r +1$.
If $L$ is projectively normal,
$L$ satisfies $(N_1^r)$ if and only if the homogeneous ideal of $X$ embedded by $|L|$ is generated by homogeneous polynomials of degree at most $r+2$.

We note that this equivalence about $(N_1^r)$  and the degrees of generators of the homogeneous ideal
does not hold in general without assuming the projectively normality.
For example, there exists a polarized abelian surface $(X,A) $ of type $(1,2)$ over the complex number field $\C$ such that
$A$ has no base divisor, 
$A^{ 2}$ is very ample, and the homogeneous ideal of $X$ embedded by $|A^{2}|$ is generated by quadrics and quartics \cite[Remark 3]{MR3656291} (see also \cite[\S2]{MR946234}).
Since $A$ has no base divisor, $A^{ 2}$ satisfies $(N_1^1)$ by \autoref{thm_Pareschi-PopaII}.
However, the homogeneous ideal of $X$ is not generated by polynomials of degree at most $3$.

In particular,
we need to assume that $A^{2}$ is projectively normal in \cite[Theorem 6.1 (b)]{MR2008719}.

\subsection{Generic vanishing, M-regularity and IT(0) of $\Q$-twisted sheaves}\label{subsec_notion_GV}

Let $l \in \mathrm{NS}(X)$ be an ample class.
For a coherent sheaf $\calf$ on $X$ and $x \in \Q$,
a \emph{$\Q$-twisted coherent sheaf} $\calf \langle xl \rangle$ is the equivalence class of the pair $(\calf,xl)$,
where the equivalence is defined by
\[
(\calf \otimes L^{m} , xl) \sim (\calf  , (x+m)l)
\]
for any line bundle $L$ representing $l$ and any $m \in \Z$.

As explained in Introduction,
a coherent sheaf $\calf$ on $X$ is said to be \emph{GV} (resp.\ \emph{M-regular}, resp.\  \emph{IT(0)})
if $\codim_{\widehat{X}} V^i(\calf) \geq i$ (resp.\ $\codim_{\widehat{X}} V^i(\calf) > i$, resp.\ $V^i(\calf) =\emptyset$) for any $i >0$,
where
$
V^i(\calf) =  \{ \alpha \in \widehat{X} \, | \,  h^i(X,\calf \otimes P_{\alpha}) >0 \}.
$

In \cite{MR4157109}, 
these notions are extended to the $\Q$-twisted setting.
A $\Q$-twisted coherent sheaf $\calf \langle xl \rangle$ for $x=\frac{a}{b}$ with $b >0$ is said to be GV, M-regular, or IT(0) 
if so is $\mu_b^* \calf \otimes L^{ab}$.
This definition does not depend on the representation $x=\frac{a}{b}$ nor the choice of $L$ representing $l$.
We note that this is true even in $\mathrm{char} (\K) \geq 0$ by \cite[Remark 3.2]{MR4114062}.
Furthermore, for an isogeny $f : Y \rightarrow X$,
\begin{align}\label{eq_pullback}
\text{$\calf \langle x l \rangle$ is GV, M-regular, or IT(0) \ $\Longleftrightarrow$ \ so is $f^*(\calf \langle x l \rangle):=f^* \calf \langle x f^*l \rangle$}
\end{align}
by \cite[Proposition 1.3.3]{CaucciThesis} or arguments in \cite[Remark 3.2]{MR4114062}.
%
%

By \cite[Theorem 5.2]{MR4157109},
we also have an equivalence
\begin{align}\label{eq_GV_IT(0)}
\text{$\calf \langle xl \rangle$ is GV \ $\Longleftrightarrow$ \
$\calf \langle (x+x')l \rangle$ is IT(0) for any rational number $x' >0$.}
\end{align}

\begin{ex}\label{ex_line bundle}
For a line bundle $B$ on $X$,
$B \langle xl \rangle$ is M-regular if and only if $B \langle xl \rangle$ is IT(0) if and only if $B+xL$ is ample by \cite[Example 3.10 (1)]{MR2435838}.
Hence 
$B \langle xl \rangle$ is GV if and only if $B+xL$ is nef.
\end{ex}

\begin{ex}\label{ex_line bundle2}
For an ample line bundle $L$ on $X$,
$\cali_o \otimes L =\cali_o \langle l \rangle$ is GV, and 
$\cali_o \otimes L $ is IT(0) if and only if $ L$ is basepoint free
by \autoref{thm_JP_Caucci} (1).
By \cite[Remark 3.6]{MR2008719},
$\cali_o \otimes L $ is M-regular if and only if $L$ has no base divisor.
\end{ex}

For a polarized abelian variety $(X,l)$,
Jiang and Pareschi introduce an invariant $\beta(l) \in \R$. 
It is defined by using cohomological rank functions, which is also introduced by them, but
$\beta(l)$ is characterized by the notion IT(0) as follows:


\begin{lem}[{\cite[Section 8]{MR4157109},\cite[Lemma 3.3]{MR4114062}}]\label{lem_beta_IT(0)}
Let $(X,l)$ be a polarized abelian variety and $x \in \Q$.
Then $\beta(l) < x$ if and only if $\cali_p \langle xl \rangle$ is IT(0) for some (and hence for any) $p \in X$.
\end{lem}

By \autoref{lem_beta_IT(0)} and \ref{eq_GV_IT(0)},
it also holds that
$\beta(l) \leq x$ if and only if $\cali_p \langle xl \rangle$ is GV for $x \in \Q$.

We note that \autoref{thm_JP_Caucci} is stated by using $\beta(l)$ in \cite{MR4157109},\cite{MR4114062},
but the original statement is equivalent to that in \autoref{thm_JP_Caucci} by \autoref{lem_beta_IT(0)}.

\begin{ex}\label{rem_beta=1/2}
If $\beta(l) < 1/2$, $L$ satisfies $(N_0)$ by \autoref{thm_JP_Caucci} (2).
Contrary to \autoref{thm_JP_Caucci} (1),
the converse does not hold in general.
For example, let $(X,A) $ be a general polarized abelian variety of type $(1,\dots,1,d)$ with $3 \leq d \leq g$ and set $L:=A^2$.
Since $h^0(A) =d \leq g$, $A$ is not basepoint free, which is equivalent to $\beta(a) =1$ by \autoref{thm_JP_Caucci} (1).
Hence we have $\beta(l) =\beta(a)/2 =1/2$.
On the other hand, $L$ is projectively normal by \cite{MR1638159}.

\end{ex}

\begin{ex}\label{ex_beta=1/2}
If $\beta(l) < 1/2$, 
$L$ satisfies not only $(N_0)$ but also $(N_1^1)$ as in the following paragraph of \autoref{thm_JP_Caucci}
(see also the proof of \autoref{thm_M-regular_is_enough} in \autoref{sec_proof of theorem}).
Hence the homogeneous ideal of $X$ embedded by $|L|$ is generated by quadrics and cubics if $\beta(l) < 1/2$.

It might be interesting to ask if the homogeneous ideal of $X$ is generated by quadrics and cubics 
under the condition that $L$ is projectively normal, which is weaker than the condition $\beta(l) < 1/2$.
At least, the answer is yes if $\dim X=2$ by \cite[Lemma 2.2]{MR3656291}.
The answer is yes as well when $L=A^2$ for some ample line bundle $A$.
In fact, $A^2$ is very ample if and only if 
$A$ has no base divisor by \cite{MR871633}.
Thus if $A^2$ is projectively normal, then $A$ has no base divisor and hence $A^2$ satisfies $(N_1^1)$ by \autoref{thm_Pareschi-PopaII} (2).
Therefore the homogeneous ideal of $X$ embedded by $|A^2|$ is generated by quadrics and cubics.
We note that \autoref{thm_Pareschi-PopaII} is true in any characteristic by \autoref{thm_M-regular_is_enough},
which we will prove in \autoref{sec_proof of theorem}.


\end{ex}

\vspace{1mm}
We will use the following proposition.

\begin{prop}[{\cite[Proposition 3.1]{MR2807853},\cite[Theorem 3.2]{MR2807853},\cite[Proposition 3.4]{MR4114062}}]\label{prop_GV+IT(0)}
Assume that $\calf$ and $\calg$ are coherent sheaves on $X$, and that one of them is locally free.
\begin{itemize}
\item[(i)] If $\calf \langle xl \rangle$ is IT(0) and $\calg \langle yl \rangle$ is GV,
then $\calf \langle xl \rangle \otimes \calg \langle yl \rangle := (\calf \otimes \calg )\langle (x+y)l \rangle$ is IT(0).
\item[(ii)] If $\calf \langle xl \rangle$ and $\calg \langle yl \rangle$ are M-regular,
then $\calf \langle xl \rangle \otimes \calg \langle yl \rangle $ is M-regular.
\end{itemize}

\end{prop}

We also note that for an isogeny $f :Y \rightarrow X$ and
a coherent sheaf $\calf$ on $Y$,
\begin{align}\label{eq_pushforward_of_F}
f_* \calf \text{ is GV, M-regular, or IT(0) } \Longleftrightarrow \text{ so is } \calf
\end{align}
since 
$
h^i(f_* \calf \otimes P_{\alpha}) =h^i(\calf \otimes f^* P_{\alpha}) =h^i(\calf \otimes P_{\hat{f}(\alpha)})
$
holds for any $\alpha \in \dX $, 
where $\hat{f} : \dX \rightarrow \widehat{Y}$ is the dual isogeny.
We see a $\Q$-twisted version of this fact in \autoref{lem_pushforward}.

\subsection{Fourier-Mukai transforms}

Let $\calp$ be the Poincar\'{e} line bundle on $X \times \dX$.
Let $D^b(X)$ be the bounded derived category of coherent sheaves on $X$ and 
\[
\Phi_{\calp} =\Phi_{\calp}^X : D^b(X) \rightarrow D^b(\dX)
\]
be the Fourier-Mukai functor associated to $\calp$.
We note that $\Phi_{\calp} (\calf)$ is a locally free sheaf (concentrated in degree $0$) 
for an IT(0) sheaf $\calf$.
For an isogeny $f : Y \rightarrow X$,
\begin{align}\label{eq_Phi_mu}
\hat{f}^* \circ \Phi_{\calp}^Y \simeq \Phi_{\calp}^X \circ f_* , \quad 
\hat{f}_* \circ  \Phi_{\calp}^{X}  \simeq \Phi_{\calp}^Y \circ  f^*
\end{align}
holds by  \cite[(3.4)]{MR607081}.



\subsection{(Skew) Pontrjagin products}\label{subsec_Pontrjagin product}

For coherent sheaves $\cale,\calf$ on $X$,
their \emph{Pontrjagin product} $\cale {*} \calf$ is defined as
\[
\cale {*} \calf = {(p_1+p_2)}_* (p_1^* \cale \otimes p_2^* \calf),
\]
where $p_i$ is the natural projection from $X \times X$ to the $i$-th factor for $i=1,2$.
By definition,
$\cale {*} \calf  =\calf {*} \cale$ holds.
Similarly,
their \emph{skew Pontrjagin product} $\cale \hat{*} \calf$ is defined as
\[
\cale  \hat{*}  \calf = {p_1}_* ((p_1+p_2)^* \cale \otimes p_2^* \calf).
\]
As in \cite[Remark 1.2]{MR1758758},
$\cale  \hat{*}  \calf  \simeq \cale {*} (-1_X)^* \calf $
and $\calf  \hat{*}  \cale  \simeq (-1_X)^* (\cale  \hat{*}  \calf  ) $,
where $-1_X=\mu^X_{-1}$. 


We will use the following properties of (skew) Pontrjagin products.
For simplicity,
we assume locally freeness or IT(0) for some sheaves.
In particular,
all the objects are sheaves and $\otimes$ is the usual (non-derived) tensor product in the following proposition.
  

\begin{prop}[{\cite[(3.7),(3.10)]{MR607081},\cite[Proposition 1.1]{MR1758758},\cite[Proposition 5.2]{MR2008719}}]\label{prop_Pontrjagin}
Let $L$ be an ample line bundle,
$\cale$ be a vector bundle and $\calf $ be a coherent sheaf on $X$.
\begin{itemize}
\item[(i)] If $ \cale, \calf$ are IT(0), then $\Phi_{\calp} ( \cale {*} \calf) = \Phi_{\calp}  (\cale) \otimes \Phi_{\calp}  (\calf)$.
\item[(ii)]  Assume that  $ h^i( (t_q^*\cale) \otimes \calf)=0  $ for any $q \in X$ and $i >0$.
For $p \in X$, the natural map
$
H^0(t_p^* \cale ) \otimes H^0(\calf) \rightarrow H^0((t_p^* \cale )\otimes \calf) 
$
is surjective if and only if 
$\cale \hat{*} \calf $ is generated by global sections at $p$.
\item[(iii)] If $L \otimes \calf $  is IT(0),
then
$L {*} \calf = L \otimes \phi_l^* \Phi_{\calp} (((-1_X)^* \calf ) \otimes L)$ and $L \hat{*} \calf = L \otimes \phi_l^* \Phi_{\calp} ( \calf \otimes L)$.  
\end{itemize}
\end{prop}

\section{Criterion for global generations}\label{sec_global_generation}


In this section, we prove \autoref{thm_criterion} (1),
which is nothing but (iii) in the following proposition.

\begin{prop}\label{lemma_globally_generated}
Let $L$ be an ample line bundle on an abelian variety $X$.
Let $\calf$ be a coherent sheaf on $X$ and let $x=a/b,y=a'/b$ be rational numbers. 
Assume that $x \geq \beta(l)$. 
\begin{itemize}
\item[(i)] If $\calf \otimes L^{-ab}$ is M-regular,
$ \cali_{ \mu_b^{-1}(p)} {\cdot} \calf$ is IT(0) for any $p \in X$,
where $\mu_b^{-1}(p) \subset X$ is the scheme-theoretic fiber over $p$
and $ \cali_{ \mu_b^{-1}(p)}{\cdot}\calf$ is the image of the natural homomorphism $  \cali_{ \mu_b^{-1}(p)} \otimes \calf \rightarrow \calf$.
\item[(ii)] If $\calf \langle yl \rangle$ is M-regular, $ (\cali_p{\cdot}\calf) \langle (x+y)l \rangle$ is IT(0) for any $p \in X$.
\item[(iii)] 
If $\calf \langle -x l \rangle $ is M-regular,
$\calf$ is globally generated. 
\end{itemize}
\end{prop}

\begin{proof}
(i) Fix $p \in X$ and
consider the exact sequence 
\[
0 \rightarrow \cali_{ \mu_b^{-1}(p)}{\cdot}\calf \otimes P_{\alpha} \rightarrow \calf \otimes P_{\alpha} 
\rightarrow \calf \otimes P_{\alpha}|_{\mu_b^{-1}(p)} \rightarrow 0
\]
for $\alpha \in \dX$.
Since $\calf \otimes L^{-ab}$ is M-regular and $ab >0$,
$\calf$ is IT(0) by \autoref{prop_GV+IT(0)}.
Since the support of $ \calf \otimes P_{\alpha}  |_{\mu_b^{-1}(p)}  $ is zero dimensional,
$h^i(\cali_{ \mu_b^{-1}(p)}{\cdot}\calf \otimes P_{\alpha}) =0$  for any $i \geq 2$
and $\cali_{ \mu_b^{-1}(p)}{\cdot}\calf$ is IT(0) if and only if the restriction map
\begin{align}\label{eq_restriciton}
H^0(\calf \otimes P_{\alpha} ) \rightarrow \calf \otimes P_{\alpha}|_{\mu_b^{-1}(p)}
\end{align}
is surjective for any $\alpha \in \dX$.
Since $\calf \otimes P_{\alpha} \otimes L^{-ab}$ is also M-regular,
it suffices to show the surjectivity of \ref{eq_restriciton} for $\alpha=o_{\dX} $.

By $x=a/b \geq \beta(l)$,
$L^{ab} \otimes \cali_{{\mu_b}^{-1}(p)} = L^{ab} \otimes \mu_b^* \cali_p$ is GV.
Hence
\[
V^1:=\{ \alpha \in \widehat{X} \, | \, h^1(L^{ab} \otimes  \cali_{{\mu_b}^{-1}(p)} \otimes P_{\alpha} ) >0 \}
\]
is a proper closed subset of $\dX$.

Since $\calf \otimes L^{-ab}$ is M-regular,
there exists an integer $N>0$ such that for any general $\alpha_1,\dots,\alpha_N \in \dX$
the natural map
\[
\bigoplus_{i=j}^N H^0(\calf \otimes L^{-ab} \otimes P_{\alpha_j}) \otimes P_{\alpha_j}^{\vee} \rightarrow \calf \otimes L^{-ab}
\]
is surjective by \cite[Proposition 2.13]{MR1949161}.
We take general $\alpha_j$ so that $-\alpha_j \not \in V^1$.

Consider the following diagram:
\[
\xymatrix{
\bigoplus_{j=1}^N H^0(\calf \otimes L^{-ab} \otimes P_{\alpha_j}) \otimes H^0(L^{ab} \otimes P_{\alpha_j}^{\vee}) \otimes \calo_X \ar[r] \ar[d] & H^0(\calf) \otimes \calo_X \ar[d] \\
\bigoplus_{j=1}^N H^0(\calf \otimes L^{-ab} \otimes P_{\alpha_j}) \otimes L^{ab} \otimes P_{\alpha_j}^{\vee} \ar[r] & \calf \\
 }
\]
Since the bottom map is surjective,
$ H^0(\calf) \rightarrow \calf|_{\mu_b^{-1} (p)}$ is surjective
if
so is
\[
\bigoplus_{j=1}^N H^0(\calf \otimes L^{-ab} \otimes P_{\alpha_j}) \otimes H^0(L^{ab} \otimes P_{\alpha_j}^{\vee}) 
\rightarrow \bigoplus_{j=1}^N H^0(\calf \otimes L^{-ab} \otimes P_{\alpha_j}) \otimes L^{ab} \otimes P_{\alpha_j}^{\vee}|_{\mu_b^{-1} (p)}.
\]
This map is surjective since 
so is $H^0(L^{ab} \otimes P_{\alpha_j}^{\vee}) \rightarrow  L^{ab} \otimes P_{\alpha_j}^{\vee}|_{\mu_b^{-1} (p)}$ for any $j$  by $-\alpha_j \not \in V^1$.
Hence (i) holds.

\vspace{1mm}
\noindent
(ii)
If $\calf \langle yl \rangle$ is M-regular,
so is $\mu_b^* \calf \otimes L^{a'b}$ by definition.
By (i), 
$ \cali_{ \mu_b^{-1}(p)}{\cdot}\mu_b^* \calf \otimes L^{a'b} \otimes L^{ab}$ is IT(0) for any $p \in X$.
Since $\mu_b$ is flat,
$ \cali_{ \mu_b^{-1}(p)}{\cdot}\mu_b^* \calf  = \mu_b^* \cali_p{\cdot}\mu_b^* \calf =\mu_b^* (\cali_p{\cdot}\calf)$ holds.
Hence $\mu_b^* (\cali_p{\cdot}\calf) \otimes L^{a'b} \otimes L^{ab}=\mu_b^* (\cali_p{\cdot}\calf) \otimes L^{(a+a')b}$ is IT(0),
which means that $( \cali_p{\cdot}\calf) \langle (x+y) l \rangle$ is IT(0).

\vspace{1mm}
\noindent
(iii)
The global generation of $\calf$ follows from the vanishing $h^1(\cali_p{\cdot}\calf )=0 $ for any $p \in X$.
Hence it suffices to show that
$\cali_p{\cdot}\calf  $ is IT(0) for any $p \in X$.
This follows from (ii) for $y=-x$.
%
\end{proof}


\section{Surjectivity of multiplication maps on global sections
}\label{sec_M_L}


Throughout this section,
$L$ is an ample line bundle on an abelian variety $X$.
If $L$ is basepoint free,
we can define a vector bundle $M_L$ on $X$ by the exact sequence
\begin{align}\label{eq_M_L}
0 \rightarrow M_L \rightarrow H^0(L) \otimes \calo_X \rightarrow L \rightarrow 0.
\end{align}


We note that we do not assume the basepoint freeness of $L$ otherwise stated.
The following proposition is essentially proved in \cite[Proposition 8.1]{MR4157109}:

\begin{prop}[{\cite[Proposition 8.1]{MR4157109}}]\label{lemma_I_o_to_M_L}
Let $\calf$ be an IT(0) sheaf $\calf$ on $X$
and $y \in \Q_{>0}$.
Then $\calf \langle -y l \rangle $ is GV, M-regular, or IT(0) if and only if so is $\phi_l^* \Phi_{\calp}(\calf) \langle \frac{1}{y} l \rangle$.

In particular,
if $L$ is basepoint free,
$\cali_o \langle x l \rangle $ is GV, M-regular, or IT(0) if and only if so is $M_L \langle \frac{x}{1-x} l \rangle$
for a rational number $ 0 < x <1$.
\end{prop}

To show \autoref{lemma_I_o_to_M_L},
we recall some results in \cite{MR4157109}.
For a coherent sheaf $\calf$ on $X$, $y \in \Q$ and $i \geq 0$,
Jiang and Pareschi define a rational number $ h^i_{\calf}(yl) \geq 0$
and study the function $\Q \rightarrow \Q : y \mapsto  h^i_{\calf}(yl)$,
which is called the \emph{cohomological rank function} of $\calf$.
See \cite{MR4157109} for the definition. 
We note that $ h^i_{\calf}(yl) $ can be defined in $\mathrm{char} (\K ) \geq 0$ by \cite[Section 2]{MR4114062}.
By \cite[Theorem 5.2 (c), Proposition 5.3]{MR4157109},
$ \calf \langle x_0 l  \rangle$ is GV (resp.\ M-regular, resp.\ IT(0)) if and only if for all $i \geq 1$ it holds that
\begin{align}\label{eq_h^i}
\text{$h^i_{\calf} ( (x_0-t) l) =O(t^i)$ (resp.\ $=O(t^{i+1})$, resp.\ $=0$) for sufficiently small $t \in \Q_{>0}$.}
\end{align}

\begin{proof}[Proof of \autoref{lemma_I_o_to_M_L}]
By \cite[Proposition 2.3]{MR4157109}, it holds that
\[
h^i_{\calf} ( -y l) = \frac{y^g}{\chi(l)} h^i_{\phi_l^* \Phi_{\calp}(\calf)}\left( \frac{1}{y} l \right)
\]
for any $y \in \Q_{>0}$
and hence 
\[
h^i_{\calf} ( (-y-t) l) = \frac{(y+t)^g}{\chi(l)} h^i_{\phi_l^* \Phi_{\calp}(\calf)}\left( \frac{1}{y+t} l \right)
\]
for any $y, t \in \Q_{>0}$.
Since 
\[
\frac{1}{y+t} =\frac{1}{y} - \frac{1}{y(y+t)} t,
\]
for a fixed $y >0$, $h^i_{\calf} ( (-y-t) l) =O(t^i)$, or $=O(t^{i+1})$, or $=0$ for sufficiently small $t  >0$
if and only if so is $h^i_{\phi_l^* \Phi_{\calp}(\calf)}\left( \frac{1}{y+t} l \right)$
if and only if so is $h^i_{\phi_l^* \Phi_{\calp}(\calf)}\left( (\frac{1}{y} - t) l \right)$.
By \ref{eq_h^i}, the first statement of this proposition holds.

When $L$ is basepoint free, $\cali_o(L) $ is IT(0) and 
$\phi_l^* \Phi_{\calp} (\cali_o(L)) =M_L \otimes L^{-1}$ as shown in the proof of \cite[Proposition 8.1]{MR4157109}.
Hence the second statement follows from the first one by considering $\calf=\cali_o(L)$ and $y=1-x$.
\end{proof}

\begin{rem}\label{rem_proj_normal_beta}
Assume that $L$ is projectively normal.
Then $H^0(L) \otimes H^0(L) \rightarrow H^0(L^2)$ is surjective
and hence $h^1(M_L \otimes L)  =0$ by \ref{eq_M_L}.
Since $ h^i(M_L \otimes L \otimes P_{\alpha})  =0$ for any $\alpha \in \dX$ and $i \geq 2$ by \ref{eq_M_L},
$M_L \otimes L$ is GV and hence 
so is $\cali_o \langle \frac12 l \rangle$ by \autoref{lemma_I_o_to_M_L},
which is equivalent to $ \beta(l) \leq 1/2$. 

More generally,
we have 
$\beta(l) \leq n/(n+1) $ if $H^0(L) \otimes H^0(L^n) \rightarrow H^0(L^{n+1})$ is surjective for $n \geq 1$ by the same argument.
For example,   if $X$ is simple and $h^0(L) > ((n+1)/n)^{g} \cdot g!$ in $\Char (\K)=0$,
this multiplication map is surjective by \cite[Theorem 1.2]{Bloss:2019aa}, and hence $\beta(l) \leq n/(n+1)$ holds.

\end{rem}

\begin{lem}\label{lem_pont_IT(0)}
Let 
$\cale$ be a locally free sheaf and
$ \calf$ be a coherent sheaf on $X$.
If $\cale$ and $ \calf$  are IT(0),
then $\cale \hat{*} \calf$ is also IT(0). 
\end{lem}

\begin{proof}
Since $\cale \hat{*} \calf \simeq (-1_X)^* (\calf \hat{*} \cale) $,
it suffices to show that $\calf \hat{*} \cale$ is IT(0). 
For $\alpha \in \dX$ and $i >0$,
we have
\begin{align*}
h^i( (\calf \hat{*} \cale) \otimes P_{\alpha}) = h^i( (\calf \hat{*} P_{\alpha}) \otimes \cale) 
=h^i(H^0(\calf \otimes P_{\alpha} ) \otimes P_{\alpha}^{\vee} \otimes \cale) =0,
\end{align*}
where the first equality follows from \cite[Proposition 5.5 (b)(i)]{MR2008719},
the second one from \cite[Remarks 3.5 (c)]{MR1758758},
and the third one holds since $\cale$ is IT(0).
We note that 
we need $ h^i(t_p^* \calf \otimes P_{\alpha} ) = h^i( t_p^* \calf  \otimes \cale) =0$ for any $p \in X$ and $i >0$ to apply \cite[Proposition 5.5 (b)(i)]{MR2008719}.
This condition is satisfied since $t_p^* \calf  $ and $ t_p^* \calf  \otimes \cale$  are IT(0).
\end{proof}

\autoref{thm_criterion} (2), (3)
follow from the following proposition:

\begin{prop}\label{prop_surjectiveity_E,F}
Let 
$\cale$ be a locally free sheaf and
$ \calf$ be a coherent sheaf on $X$.
Assume 
\begin{itemize}
\item[(a)] $\cale,\calf$ are IT(0), and
\item[(b)]  $ \phi_l^* \Phi_{\calp}(\cale)  \otimes \phi_l^* \Phi_{\calp} ((-1_X)^*\calf )   \left\langle  \frac{1}{x} l \right\rangle  $ is M-regular for some rational number $x \geq \beta(l)$.
\end{itemize}
Then the natural map
$H^0(t_p^* \cale ) \otimes H^0(\calf) \rightarrow H^0(t_p^* \cale \otimes \calf)$ is surjective for any $p \in X$.

Furthermore,
the assumptions (a), (b) are satisfied if 
\begin{enumerate}
\item $\cale=L$, and  $\calf \langle -\frac{x}{1-x} l \rangle$ is M-regular for some $\beta(l) \leq x <1$, or
\item there exist rational numbers $ s,t >0$ such that 
$\cale \langle -s l\rangle$ and $ \calf \langle -t l\rangle$ are M-regular and $st/(s+t) \geq \beta(l)$.
\end{enumerate}
\end{prop}

\begin{proof}
Assume that (a), (b) are satisfied.
To show the surjectivity of $H^0(t_p^* \cale ) \otimes H^0(\calf) \rightarrow H^0(t_p^* \cale \otimes \calf)$ for any $p \in X$,
it suffices to show that $\cale \hat{*} \calf  $ is globally generated by \autoref{prop_Pontrjagin} (ii). 
By \autoref{thm_criterion} (1),
it is enough to show the M-regularity of $\cale \hat{*} \calf \langle -x l \rangle$.
Since $\cale, \calf$ are IT(0), 
so is $\cale \hat{*} \calf $ by \autoref{lem_pont_IT(0)}.
Hence 
$\cale \hat{*} \calf \langle -x l \rangle$ is M-regular if and only if 
so is $\phi_l^* \Phi_{\calp}(\cale \hat{*} \calf ) \langle \frac{1}{x} l \rangle$ by  \autoref{lemma_I_o_to_M_L}.
Since  
\begin{align*}\label{eq_phi_pont}
 \Phi_{\calp}(\cale \hat{*} \calf ) = \Phi_{\calp}(\cale {*} (-1_X)^*\calf ) = \Phi_{\calp}(\cale) \otimes \Phi_{\calp} ((-1_X)^*\calf ) 
\end{align*}
by  \autoref{prop_Pontrjagin} (i),
we have $\phi_l^* \Phi_{\calp}(\cale \hat{*} \calf ) \langle \frac{1}{x} l \rangle =\phi_l^* \Phi_{\calp}(\cale)  \otimes \phi_l^* \Phi_{\calp} ((-1_X)^*\calf )  \langle \frac{1}{x} l \rangle $,
which is M-regular by (b).
Hence $\cale \hat{*} \calf \langle -x l \rangle$ is M-regular and 
$H^0(t_p^* \cale ) \otimes H^0(\calf) \rightarrow H^0(t_p^* \cale \otimes \calf)$ is surjective for any $p \in X$

\vspace{1mm}
The rest is to show that (a), (b) are satisfied in the cases (1), (2).

\vspace{1mm}
\noindent
(1) Since $x / (1-x) >0$ and $\calf \langle \frac{-x}{1-x} l \rangle$ is M-regular, $\calf $ is IT(0).
Since $L$ is an ample line bundle and hence IT(0), (a) is satisfied. 

Since
$\phi_l^* \Phi_{\calp}(\cale)  = \phi_l^* \Phi_{\calp}(L) =H^0(L) \otimes L^{-1}  $ by \cite[Proposition 3.11 (1)]{MR607081},
we have 
\[
\phi_l^* \Phi_{\calp}(\cale)  \otimes \phi_l^* \Phi_{\calp} ((-1_X)^*\calf )   \left\langle  \frac{1}{x} l \right\rangle 
=H^0(L) \otimes \phi_l^* \Phi_{\calp} ((-1_X)^*\calf )   \left\langle  \frac{1-x}{x} l \right\rangle .
\]
By \autoref{lemma_I_o_to_M_L},
this is M-regular if and only if so is $((-1_X)^*\calf ) \langle  -\frac{x}{1-x} l \rangle$,
which is equivalent to the M-regularity of $\calf  \langle  -\frac{x}{1-x} l \rangle$ by $(-1_X)^*l=l$.
Hence (b) is satisfied.

\vspace{1mm}
\noindent
(2) Since $s, t >0$ and $\cale \langle -s l\rangle$ and $ \calf \langle -t l\rangle$ are M-regular,
(a) is satisfied.

For (b), set $x=st/(s+t) \geq \beta(l)$.
Then $1/x=1/s+1/t $ holds and 
\[
\phi_l^* \Phi_{\calp}(\cale)  \otimes \phi_l^* \Phi_{\calp} ((-1_X)^*\calf )   \left\langle  \frac{1}{x} l \right\rangle 
= \left(\phi_l^* \Phi_{\calp}(\cale)  \left\langle  \frac{1}{s} l \right\rangle \right) 
\otimes \left(\phi_l^* \Phi_{\calp} ((-1_X)^*\calf) \left\langle \frac{1}{t}  l \right\rangle  \right) .
\]
By \autoref{lemma_I_o_to_M_L},
$ \phi_l^* \Phi_{\calp}(\cale)  \left\langle  \frac{1}{s} l \right\rangle $ and 
$ \phi_l^* \Phi_{\calp} ((-1_X)^*\calf) \left\langle \frac{1}{t}  l \right\rangle $ are M-regular
since so are $\cale \langle -s l\rangle$ and $((-1_X)^*\calf ) \langle -t l\rangle$.
Hence (b) is satisfied by \autoref{prop_GV+IT(0)}.
\end{proof}

\begin{rem}\label{rem_relation (2), (3)}
The M-regularity of 
$L \hat{*} \calf \langle  -x l \rangle$ in the case (1) of \autoref{prop_surjectiveity_E,F} can be proved slightly easier as follows:

By \autoref{prop_Pontrjagin} (iii), we have
$L \hat{*} \calf \langle  -x l \rangle   = L \otimes \phi_l^* \Phi_{\calp}(\calf \otimes L) \langle  -x l \rangle  =  \phi_l^* \Phi_{\calp}(\calf \otimes L) \langle  (1-x) l \rangle$.
By \autoref{lemma_I_o_to_M_L},
this is M-regular if and only if so is $ \calf \otimes L \langle -\frac{1}{1-x} l\rangle = \calf  \langle -\frac{x}{1-x} l\rangle$.
\end{rem}

\autoref{lemma_I_o_to_M_L} also implies
the following lemma,
which is a $\Q$-twisted version of \ref{eq_pushforward_of_F}.

\begin{lem}\label{lem_pushforward}
Let $f : Y \rightarrow X$ be an isogeny 
and $\calf$ be a coherent sheaf on $Y$.
For $x \in \Q$,
$ f_* \calf \langle x l \rangle $ is GV, M-regular or IT(0) if and only if so is $\calf \langle x f^* l \rangle$.

In particular,
for a coherent sheaf $\calf$ on $X$, an integer $m >0$ and $x \in \Q$,
${\mu_m}_* \calf \langle xl\rangle $ is GV, M-regular or IT(0) if and only if so is $\calf \langle m^2xl\rangle $.
\end{lem}

\begin{proof}
For a sufficiently large integer $ n $,
$\calf \otimes f^* L^n $ is IT(0). 
Since
\[
f_* \calf \langle xl\rangle =f_* (\calf \otimes f^* L^n) \langle (x-n)l\rangle , \quad \calf \langle x f^*l \rangle  = \calf \otimes f^* L^n  \langle (x-n) f^*l\rangle ,
\]
we may assume that $\calf$ is IT(0) and $x < 0$
by replacing $\calf $ and $x$ with $\calf \otimes f^* L^n $ and $x-n$ respectively.

Set $y:=-x >0$.
By \autoref{lemma_I_o_to_M_L} and \ref{eq_pullback},
${f}_* \calf \langle xl\rangle = {f}_* \calf \langle -y l\rangle $ is GV, M-regular or IT(0) if and only if 
so is $\phi_l^* \Phi_{\calp}^X ({f}_* \calf  ) \langle \frac{1}{y} l \rangle $
if and only if so is $ f^* \phi_l^* \Phi_{\calp}^X ({f}_* \calf  ) \langle \frac{1}{y} f^* l \rangle $.
We have
\begin{align*}
f^* \phi_l^* \Phi_{\calp}^X ({f}_* \calf  )
= f^* \phi_l^* \hat{f}^* \Phi_{\calp}^Y ( \calf  )
=  \phi_{f^* l}^*  \Phi_{\calp}^Y (\calf  )
\end{align*}
by \ref{eq_Phi_mu} and $ \hat{f} \circ \phi_l \circ f =\phi_{f^* l}$.
Hence
${f}_* \calf \langle xl\rangle $ is GV, M-regular or IT(0) if and only if 
so is $  \phi_{f^* l}^*  \Phi_{\calp}^Y (\calf  ) \langle \frac{1}{y} f^* l \rangle $ if and only if so is $\calf \langle -y f^*l \rangle=\calf \langle x f^*l \rangle$
by \autoref{lemma_I_o_to_M_L}.

The last statement is nothing but the spacial case $f=\mu_m$ since $\mu_m^* l=m^2 l$.
\end{proof}

\section{On jet ampleness}\label{sec_jet-ample}

In this section, we study $k$-jet ampleness using \autoref{lemma_globally_generated}.
Throughout this section,
$L$ is an ample line bundle on an abelian variety $X$.
First, we prove \autoref{thm_jet-ample}.

\begin{proof}[Proof of \autoref{thm_jet-ample}]
If $ \cali_o \langle \frac{1}{k+1} l \rangle$ is M-regular, so is $  \cali_{p_1} \langle \frac{1}{k+1} l \rangle$ for any $p_1 \in X$.
Furthermore, we have $\beta(l) \leq \frac{1}{k+1}$.
Applying \autoref{lemma_globally_generated} (ii) to $x=\frac{1}{k+1}$ and $ \calf \langle yl\rangle =\cali_{p_1} \langle \frac{1}{k+1} l \rangle$,
we see that $ \cali_{p_2}  \cali_{p_1} \langle \frac{2}{k+1} l \rangle$ is IT(0) for any $p_2 \in X$.
We note that $p_2$ can be $p_1$.
Since IT(0) implies M-regularity,
we can apply \autoref{lemma_globally_generated} (ii) to $ \cali_{p_2}  \cali_{p_1} \langle \frac{2}{k+1} l \rangle$
and hence
$ \cali_{p_3}  \cali_{p_2} \cali_{p_1} \langle \frac{3}{k+1} l \rangle$ is IT(0) for any $p_3 \in X$.
Repeating this,
we obtain that $ \cali_{p_{k+1}} \cdots \cali_{p_2} \cali_{p_1} \langle \frac{k+1}{k+1} l \rangle$ is IT(0),
i.e.\ $ \cali_{p_{k+1}} \cdots \cali_{p_2} \cali_{p_1} \otimes L$ is IT(0) for any (not necessarily distinct) $p_1,\dots,p_{k+1} \in X$.
By \cite[Lemma 3.3]{MR2008719},
this is equivalent to the $k$-jet ampleness of $L$.
\end{proof}

In \cite{MR2008719}, the authors introduce an invariant $m(L)$, called the \emph{M-regularity index}, as 
\begin{align*}
m(L) :=\max\{ m \geq 0 \, | \, &L \otimes \cali_{p_1}\cali_{p_2} \cdots \cali_{p_m} \text{ is M-regular} \\ 
&\hspace{10mm} \text{ for any (not necessarily distinct) $m$ points } p_1,\dots,p_m \in X   \}.
\end{align*}
\cite[Theorem 3.8]{MR2008719} states that if $A,L_1,\dots,L_{k+1-m(A)}$ are ample line bundles on $X$ and $k \geq m(A) \geq 1$,
then $A \otimes L_1 \otimes \dots \otimes L_{k+1-m(A)}$ is $k$-jet ample.
In particular, $A^{k+2-m(A)}$ is $k$-jet ample for $k \geq m(L) \geq 1$.
The following is a generalization of this result:

\begin{prop}\label{prop_jet_ampleness2}
Let $n,k_1,\dots,k_n $ be positive integers and
$A,L_1, \dots, L_n$ be ample line bundles on $X$. 
If $\beta(l_i) \leq \frac{1}{k_i}$ for any $1 \leq i \leq n$,
then $A \otimes L_1 \otimes \dots \otimes L_n$ is $k$-jet ample,
where $k= m(A) + \sum_{i=1}^n k_i -1$.
\end{prop}

\begin{proof}
For simplicity, set $m=m(A)$.
By definition,
$A \otimes \cali_{p_1}\cali_{p_2} \cdots \cali_{p_m}$ is M-regular for any $p_i \in X$.
By \autoref{lemma_globally_generated} (ii),
$A \otimes \cali_{p_1}\cali_{p_2} \cdots \cali_{p_m} \cali_{p_{m+1}} \langle \frac{1}{k_1} l_1 \rangle$ is IT(0) for any $p_{m+1} \in X$.
Since IT(0) implies M-regularity,
$A \otimes \cali_{p_1}\cali_{p_2} \cdots \cali_{p_m} \cali_{p_{m+1}}\cali_{p_{m+2}}  \langle \frac{2}{k_1} l_1 \rangle$ is IT(0) for any $p_{m+2} \in X$ by \autoref{lemma_globally_generated} (ii).
Repeating this,
$A \otimes \cali_{p_1}\cali_{p_2} \cdots \cali_{p_{m+k_1}}  \langle \frac{k_1}{k_1} l_1 \rangle = A \otimes L_1 \otimes \cali_{p_1}\cali_{p_2} \cdots \cali_{p_{m+k_1}} $ is IT(0)  for any $p_i$.
Repeating this argument,
$A \otimes L_1 \otimes \dots \otimes L_n \otimes  \cali_{p_1}\cali_{p_2} \cdots \cali_{p_{m+\sum{k_i}}}
= A \otimes L_1 \otimes \dots \otimes L_n \otimes  \cali_{p_1}\cali_{p_2} \cdots \cali_{p_{k+1}}$ is IT(0) for any $p_i \in X $.
By \cite[Lemma 3.3]{MR2008719},
this is equivalent to the $k$-jet ampleness of  $A \otimes L_1 \otimes \dots \otimes L_n$.
\end{proof}

We note that $A \otimes L_1 \otimes \dots \otimes L_n$ in \autoref{prop_jet_ampleness2} is a tensor product of two or more ample line bundles
since we assume $n \geq 1$.
Hence \autoref{prop_jet_ampleness2} do not contain \autoref{thm_jet-ample}. 

\vspace{2mm}
In the rest of this section, 
we see some relation between $m(L) $, the M-regularity of $\cali_o \langle x l \rangle$,
and Seshadri constants.

\begin{cor}\label{cor_m(L)}
If $\cali_o \langle \frac{1}{m} l \rangle$ is M-regular for an integer $ m \geq 1$,
then $m(L) \geq m$  holds.
\end{cor}

\begin{proof}
If $m=1$,
this holds from definition.
Hence we may assume that $m \geq 2$.
Then we already see that $ \cali_{p_{m}} \cdots \cali_{p_2} \cali_{p_1} \otimes L$ is IT(0) for any $p_1,\dots,p_{m} \in X$ in the proof of \autoref{thm_jet-ample}.
Hence $m(L) \geq m $ holds.
\end{proof}

In general, the converse of \autoref{cor_m(L)} does not hold,
that is, $\cali_o \langle \frac{1}{m(L)} l \rangle$ is not M-regular in general.
For example,
let $(X,L)$ be a general polarized abelian surface of type $(1,4)$.
By \cite[Example 3.7]{MR2008719},
we have $m(L)=2$.
On the other hand,
$\cali_o \langle \frac{1}{2} l \rangle$  is not M-regular by \autoref{thm_jet-ample} since $ L$ is not very ample.

In fact, $\cali_o \langle \frac{1}{m(L)} l \rangle$ is not even GV, equivalently $\beta(l) \leq 1/m(L)$ does not hold in general.
For example,
let $(X,L)$ be a polarized abelian surface of type $(1,d)$ with Picard number one.
Then $L$ is $k$-jet ample if and only if $d > \frac12 (k+2)^2$ by \cite{MR1376538}.
Hence we have $m(L) \geq k+1$ if $d > \frac12 (k+2)^2$ by \cite[Proposition 3.5]{MR2008719}.
Thus $m(L) \geq \lfloor \sqrt{2d-1} \rfloor -1$.
On the other hand,
$\beta(l) \geq 1/\sqrt{d}$ holds by \cite[Lemma 3.4]{Ito:2020aa}.
Hence $1/m(L) $, which is bounded from above by roughly $1/\sqrt{2d}$, is strictly smaller than $\beta(l)$ for $d \gg 1$.

However, we can show that $\cali_o \langle \frac{\dim X}{m(L)} l \rangle$ is IT(0) at least in $\mathrm{char} (\K) =0$.
To see this,
recall that the \emph{Seshadri constant} $\ep(L) $ of $L$ is defined as
\[
\ep(X,L)=\ep(L)  :=\max\{  t \geq 0 \, | \, \pi^* L -tE \text{ is nef} \},
\]
where $\pi :\tilde{X} \rightarrow  X$ is the blow-up at $o \in X$ and $E \subset \tilde{X} $ is the exceptional divisor.
By \cite[Theorem 7.4]{MR2807853},
$\ep(L) \geq m(L)$ holds.
On the other hand,
$\beta(l) \cdot \ep(L) \geq 1$ holds by \cite[Proposition 1.6.7]{CaucciThesis},
that is, $\ep(L) > x^{-1}$ holds if 
$ \cali_o  \langle x l \rangle$ is IT(0).
The following is a refinement of these results:


\begin{prop}\label{lem_seshadri}
Assume $\mathrm{char} (\K) =0$.
\begin{enumerate}
\item[(i)] If $ \cali_o^m  \langle x l \rangle$ is M-regular for an integer $m\geq 0$ and $x \in \Q_{>0}$,
then 
$ \ep(L) > x^{-1} \cdot m $ holds.
\item[(ii)] If $ \cali_o^m  \langle x l \rangle$ is GV for an integer $m\geq 0$ and $x \in \Q_{>0}$,
then 
$ \ep(L)  \geq x^{-1} \cdot m $ holds.
\item[(iii)] 
$\ep(L) > m(L)$ holds.
\item[(iv)]
If $m(L) \geq 1$,
$\cali_o \langle \frac{g}{m(L)} l \rangle$ is IT(0),
where $g=\dim X$.
\end{enumerate}
\end{prop}

\begin{proof}
(i) 
Since $\ep(L) >0$, this is clear if $m=0$.
Hence we may assume $m \geq 1$.
Let $\pi :\tilde{X} \rightarrow  X$ be the blow-up at $o \in X$ and $E \subset \tilde{X} $ be the exceptional divisor.
What we need to show is the ampleness of $\pi^* L -  x^{-1} \cdot m E$, i.e.\ the ampleness of $ \pi^* xL -   m E$.

Let $x=a/b$.
Then $\mu_b^* \cali_o^m \otimes L^{ab}$ is M-regular.
Hence $\mu_b^* \cali_o^m \otimes L^{ab}$ is ample,
that is, the tautological line bundle $\calo(1)$ of 
\[
 \P_X( \mu_b^* \cali_o^m \otimes L^{ab}) :=\Proj_X \left( \bigoplus_{n \geq 0} \Sym^n  (\mu_b^* \cali_o^m \otimes L^{ab}) \right)
\]
is ample
by \cite[Proposition 2.13]{MR1949161} and \cite[Corollary 3.2]{MR2233707}.
Since 
\[
\Proj_X \left( \bigoplus_{n \geq 0} \Sym^n  (\mu_b^* \cali_o^m \otimes L^{ab}) \right) 
\simeq \Proj_X \left( \bigoplus_{n \geq 0} \Sym^n  (\mu_b^* \cali_o^m ) \right)
\simeq \Proj_X \left( \bigoplus_{n \geq 0} \Sym^n  (\mu_b^* \cali_o ) \right),
\]
$ \P_X( \mu_b^* \cali_o^m \otimes L^{ab}) \rightarrow X$ is isomorphic to the blow-up $\pi' : \tilde{X}' \rightarrow X$ along the ideal $\mu_b^* \cali_o $.
Under this isomorphism,
the tautological line bundle $\calo(1)$ on $ \P_X( \mu_b^* \cali_o^m \otimes L^{ab})$ corresponds to $ \calo(-mE') \otimes \pi'^* L^{ab}$,
where $E' \subset \tilde{X}' $ is the exceptional divisor of $\pi'$.
Hence $  \pi'^* ab L  -mE'  $ is ample.

Since $\pi'$ is the blow-up along $\mu_b^* \cali_o $,
there exists a morphism $ \tilde{\mu}_b :  \tilde{X}'  \rightarrow \tilde{X}$ such that 
$\pi \circ \tilde{\mu}_b = \mu_b \circ \pi' $ and $\tilde{\mu}_b^* \calo(-E) = \calo(-E')$.
Hence we have $  \pi'^* ab L  -mE' \equiv \tilde{\mu}_b^*  (  \pi^* xL -   m E)$
and the ampleness of $\pi^* xL -   m E$ follows from that of $  \pi'^* ab L  -mE'  $. 

\vspace{1mm}
\noindent
(ii) 
If $\cali_o^m \langle  xl \rangle$ is GV, then $\cali_o^m \langle  (x+x' ) l \rangle$ is IT(0) for any rational number $x' >0 $.
Hence $\ep(L) > (x+x')^{-1} \cdot m$ holds by (1).
By $x' \rightarrow 0$, we have $\ep(L) \geq x^{-1} \cdot m$.

\vspace{1mm}
\noindent
(iii) 
By definition, $\cali_o^{m(L)} \langle l \rangle$ is M-regular
and hence $\ep(L) > m(L)$ holds by (1).

\vspace{1mm}
\noindent
(iv)
By \cite[Proposition 3.1]{Ito:2020aa} and (3),
we have $\beta(l)  \leq g \cdot \ep(L)^{-1} < g \cdot m(L)^{-1}$.
Hence (4) follows from \autoref{lem_beta_IT(0)}.
\end{proof}

\begin{ex}
(1) By definition, $\cali_o^2 \langle l \rangle$ is 
GV (resp.\ M-regular) if and only if the codimension of 
$
\{ p \in X \,  | \,   h^1(X, \cali_p^2 \otimes L)  > 0 \}
$
in $X$ is at least one (resp.\ at least two).
Since $h^1(X, \cali_p^2 \otimes L)  =0 $ holds if and only if the rational map $f_{|L|}$ defined by $|L|$ is an immersion at $p$,
we have $\ep(L) \geq 2 $ if $ f_{|L|} $ is generically finite, and $\ep(L) > 2 $ if $f_{|L|}$ is an immersion outside a codimension two subset
by \autoref{lem_seshadri} (i), (ii)
in $\Char (\K)=0$.

\noindent
(2)
In \cite[Theorem 1.1, Lemma 2.6]{MR1393263},
Nakamaye proves that for a polarized abelian variety $(X,L)$ in $ \Char(\K)=0$, $\ep(L) \geq 1$ and the equality holds if and only if 
\begin{align}\label{eq_decomp}
 (X,L) \simeq (E, L_E) \times  (X', L') :=(E \times X', p_E^* L_E \otimes p_{X'}^* L')
\end{align}
for a principally polarized elliptic curve $(E,L_E)$ and a polarized abelian variety $(X',L')$.
We can recover this result from \autoref{lem_seshadri} as follows:

We show this by the induction on $g=\dim X$.
If $g=1$, this is clear since $\ep(X,L) = \deg (L)$ by definition.
Assume $g \geq 2$.

If $(X,L) \simeq (E, L_E) \times  (X', L') $ as \ref{eq_decomp}, then we have  
$ \ep(X,L) =\min \{ \ep(E,L_E), \ep(X',L')\}$ (see \cite[Proposition 3.4]{MR3338009} for example).
Since $\ep(E,L_E)=\deg L_E=1 $ 
and $ \ep(X',L') \geq 1$ by induction hypothesis,
we have $ \ep(X,L) =1$.

Conversely, assume $\ep(L) \leq 1$.
Then 
$ m(L) =0$ by \autoref{lem_seshadri} (ii)
and hence $L$ has a base divisor.

Assume that $(X,L)$ is indecomposable,
that is, $(X,L) $ is not isomorphic to any product 
of polarized abelian varieties of positive dimensions.
Then $L $ is an indecomposable principal polarization by 
\cite[Theorem 4.3.1]{MR2062673}
since $L$ has a base divisor.
Replacing $L$ with an algebraically equivalent line bundle,
we may assume that $L $ is symmetric.
Then the morphism $f_{|L^2|}$ defined by $|L^{2}|$ is the natural morphism 
$\pi : X \rightarrow X/\langle -1_X\rangle \subset \P^{2^g-1}$  to the Kummer variety $ X/\langle -1_X \rangle $ 
by \cite[Theorem 4.8.1]{MR2062673}.
Since $\pi$ is an immersion outside two-torsion points in $X$ and $g \geq 2$,
we have $2 \ep(L)=\ep(L^2) > 2$ by (1) of this example.
This contradicts the assumption $ \ep(L) \leq 1$.

Thus $(X,L)$ is decomposable, that is, $(X,L) $ is isomorphic to a product $ (X_1,L_1) \times (X_2,L_2)$ with $\dim X_i \geq 1$.
Then $\ep(X,L)= \min\{ \ep(X_1,L_1) ,\ep(X_2,L_2)\}$ 
and hence we may assume $ \ep(X_1,L_1) =\ep(X,L) \leq 1$.
Since $\ep(X_1,L_1)  \geq 1$ by induction hypothesis, $ \ep(X_1,L_1) =1$ holds and 
$(X_1,L_1) $ has a decomposition as \ref{eq_decomp}.
Hence $\ep(X,L) =\ep(X_1,L_1) =1 $ and $(X,L)$ has a decomposition as \ref{eq_decomp}.



\end{ex}

\section{On property $(N_p)$}\label{sec_proof of theorem}

Throughout this section,
$L$ is an ample line bundle on an abelian variety  $X$.
If $\cali_o \langle x l \rangle$ is M-regular for some rational number $0< x <1$,
then $L$ is basepoint free by \autoref{thm_JP_Caucci} (1).
Furthermore,
$M_L \langle \frac{x}{1-x} l \rangle$ is M-regular by \autoref{lemma_I_o_to_M_L}.
Hence $M_L^{\otimes m} \langle \frac{mx}{1-x} l \rangle$ is also M-regular for any $m \geq 1$ by \autoref{prop_GV+IT(0)}.
In fact, we can show that $M_L^{\otimes m} \langle \frac{mx}{1-x} l \rangle$ is IT(0) if $m \geq 2$ as follows: 

\begin{prop}\label{lem_M^2}
Assume that $\cali_o \langle x l \rangle$ is M-regular for a rational number $0 < x <1$.
Then $M_L^{\otimes m} \langle \frac{mx}{1-x} l \rangle$ is IT(0) for any integer $m \geq 2$.
\end{prop}

\begin{proof}
Let $1-x=\frac{a}{b}$ for integers $b > a > 0$. 
Then 
\begin{align}\label{eq_M-regular}
M_L  \left\langle \frac{x}{1-x} l \right\rangle  =M_L \left \langle \frac{b-a}{a} l \right \rangle \  \text{ and } \ \mu_a^* M_L \otimes L^{a(b-a)} \ \text{ are M-regular}
\end{align}
by \autoref{lemma_I_o_to_M_L}.
If $M_L^{\otimes 2} \langle \frac{2x}{1-x} l \rangle$ is IT(0),
so is 
\[
M_L^{\otimes m} \left \langle \frac{mx}{1-x} l \right \rangle = 
M_L^{\otimes 2} \left  \langle \frac{2x}{1-x} l  \right \rangle \otimes \left(M_L \left \langle \frac{x}{1-x} l \right \rangle \right)^{\otimes m-2}
\]
for $m \geq 2$
by \autoref{prop_GV+IT(0)}
since $M_L^{\otimes 2} \langle \frac{2x}{1-x} l \rangle$ is IT(0) and $M_L \langle \frac{x}{1-x} l \rangle  $  is GV.
Hence it suffices to show the case $m=2$.
\vspace{1mm}

By definition and  \ref{eq_pushforward_of_F},
$M_L^{\otimes 2} \langle \frac{2x}{1-x} l \rangle= M_L^{\otimes 2} \langle \frac{2(b-a)}{a} l \rangle$ is IT(0) if and only if so is
$\mu_{a}^* M_L^{\otimes 2} \otimes L^{2a(b-a)}$
if and only if so is
${\mu_{a}}_* (\mu_{a}^* M_L^{\otimes 2} \otimes L^{2a(b-a)}) = M_L^{\otimes 2} \otimes {\mu_{a}}_* L^{2a(b-a)}$.
Consider the exact sequence
\begin{align*}\label{eq_pullback}
0 \rightarrow M_L^{\otimes 2} \otimes {\mu_{a}}_* L^{2a(b-a)}  \rightarrow H^0(L) \otimes M_L \otimes {\mu_{a}}_* L^{2a(b-a)}
\rightarrow L \otimes M_L \otimes {\mu_{a}}_* L^{2a(b-a)} \rightarrow 0
\end{align*}
obtained by tensoring $M_L \otimes {\mu_{a}}_* L^{2a(b-a)}$ with \ref{eq_M_L}.
Since $\mu_a^* M_L \otimes L^{a(b-a)} $ is M-regular by \ref{eq_M-regular},
$\mu_a^* M_L \otimes L^{2a(b-a)} $ is IT(0).
Hence $M_L \otimes {\mu_{a}}_* L^{2a(b-a)} ={\mu_{a}}_* ( \mu_{a}^* M_L \otimes L^{2a(b-a)}) $ and $L \otimes M_L \otimes {\mu_{a}}_* L^{2a(b-a)}$ are IT(0) by \ref{eq_pushforward_of_F} and \autoref{prop_GV+IT(0)}.
Thus $h^i( M_L^{\otimes 2} \otimes {\mu_{a}}_* L^{2a(b-a)} \otimes P_{\alpha}) =0$ for any $i \geq 2$ and $\alpha \in \widehat{X}$,
and $M_L^{\otimes 2} \otimes {\mu_{a}}_* L^{2a(b-a)}$ is IT(0) if and only if 
the natural map
\[
H^0(L) \otimes H^0(M_L \otimes {\mu_{a}}_* L^{2a(b-a)} \otimes P_{\alpha})
\rightarrow H^0(L \otimes M_L \otimes {\mu_{a}}_* L^{2a(b-a)} \otimes P_{\alpha})
\]
is surjective for any $\alpha \in \dX$.
By \autoref{thm_criterion} (2),
this map is surjective if
\begin{align*}
M_L \otimes {\mu_{a}}_* L^{2a(b-a)} \otimes P_{\alpha} \left \langle -\frac{x}{1-x} l \right \rangle 
&=M_L \otimes {\mu_{a}}_* L^{2a(b-a)} \otimes P_{\alpha} \left  \langle -\frac{b-a}{a} l \right  \rangle \\
&={\mu_a}_* (\mu_a^* M_L \otimes L^{2a(b-a)} \otimes \mu_a^* P_{\alpha} )  \left \langle -\frac{b-a}{a} l \right  \rangle
\end{align*}
is M-regular.
By \autoref{lem_pushforward},
this is equivalent to the M-regularity of 
\[
\mu_a^* M_L \otimes L^{2a(b-a)} \otimes  \mu_a^* P_{\alpha}  \left  \langle - a^2\frac{b-a}{a} l \right\rangle = \mu_a^* M_L   \left  \langle a(b-a) l \right\rangle,
\]
which is nothing but \ref{eq_M-regular}. 
\end{proof}


\begin{prop}\label{thm_vanishing}
Assume that $\cali_o \langle x l \rangle$ is M-regular for a rational number $0 < x <1$.
Then $h^i( M_L^{\otimes m} \otimes B)=0$ for an integer $m \geq 1$ and a line bundle $B$ on $X$
if $ B - \frac{mx}{1-x}L$ is ample or $m \geq 2$ and $ B - \frac{mx}{1-x}L$ is nef.

In particular,
if  $ \cali_o \langle \frac{1}{p+2} l \rangle$ is M-regular for an integer $p \geq 0$,
then $h^i( M_L^{\otimes m} \otimes L^h)=0$ for positive integers $m,h $
if 
$h > m/(p+1)$ or $m \geq 2$ and $h \geq m/(p+1)$.
\end{prop}

\begin{proof}
The proof is essentially the same as that of \cite[Proposition 3.5]{MR4114062}
other than we use \autoref{lem_M^2} when $m \geq 2$.

As a $\Q$-twisted sheaf,
$M_L^{\otimes m} \otimes B$ is written as
\[
\left(M_L  \left \langle \frac{x}{1-x} l \right \rangle \right)^{\otimes m} \otimes B \left \langle - \frac{mx}{1-x} l \right \rangle.
\]
By \autoref{lemma_I_o_to_M_L},
$M_L   \langle \frac{x}{1-x} l \rangle$ is M-regular.

If $ B - \frac{mx}{1-x}L$ is ample, $B \langle - \frac{mx}{1-x} l \rangle$ is IT(0) by \autoref{ex_line bundle}.
Thus $M_L^{\otimes m} \otimes B$ is IT(0) by \autoref{prop_GV+IT(0)}
and hence the vanishing $h^i( M_L^{\otimes m} \otimes B)=0$ holds.

When $m \geq 2$ and $ B - \frac{mx}{1-x}L$ is nef,
$\left(M_L   \langle \frac{x}{1-x} l  \rangle \right)^{\otimes m}$ is IT(0) by \autoref{lem_M^2}
and $B \langle - \frac{mx}{1-x} l \rangle$ is GV by \autoref{ex_line bundle}.
Hence $M_L^{\otimes m} \otimes B$ is IT(0) by \autoref{prop_GV+IT(0)}
and the vanishing $h^i( M_L^{\otimes m} \otimes B)=0$ holds.

The last statement is just a special case when $x=1/(p+2) $ and $B=L^h$.
\end{proof}

\begin{proof}[Proof of \autoref{thm_M-regular_is_enough}]
(1) Since $\beta(l) \leq 1/(p+2) <1$, $L$ is basepoint free.
By \cite[Proposition 4.1]{MR4114062},
$(N_p)$ holds for $L$
if $ h^1(M_L^{\otimes p+1} \otimes L^h)=0$ for any $h \geq 1$ in $\Char(\K) \geq 0$.
Since $p+1 \geq 2$,
here we use the assumption $p \geq 1$, and $h \geq  1 =(p+1)/(p+1)$,
this vanishing follows from \autoref{thm_vanishing}.



For (2),
$(N_p^r)$ holds for $L$
if $ h^1(M_L^{\otimes p+1} \otimes L^h)=0$ for any $h \geq r+1$ by \cite[Proposition 6.3]{MR2008719} when $\mathrm{char} (\K)$ does not divide $p+1$ and
by \cite[Section 4]{MR4114062} in any characteristic.
We can prove (2) similarly.
\end{proof}

We give an example which does not follow from \autoref{thm_Pareschi-PopaII}, \autoref{thm_JP_Caucci}.

\begin{ex}\label{ex_(1,3,...,3)}
Let $(X,L)$ be a general polarized abelian variety of dimension $g \geq 2$ and of type $(1,3,\dots,3)$ in $\mathrm{char}(\K) =0$.
Then
$L^n$ satisfies $(N_p) $ 
if $ n \geq \frac{2(p+2)}{3}$.
To see this,
it suffices to show that $\cali_o \langle \frac23 l\rangle $ is M-regular by \autoref{thm_M-regular_is_enough}.
In fact,
we can show that $\cali_o \langle \frac23 l\rangle $ is M-regular  but not IT(0) as follows:

Take an isogeny $\pi : Y \rightarrow X$ with kernel $\pi^{-1}(o) \simeq \Z/3\Z$ such that there exists a principally polarization $\Theta$ 
such that $ \pi^* L \equiv 3 \Theta$.
Then  $\cali_o \langle \frac23 l\rangle $ is M-regular or IT(0) 
if and only if 
so is $\pi^* \cali_o \langle \frac23 \pi^* l \rangle =\cali_{\pi^{-1}(o)} \langle 2 \theta \rangle$ by \ref{eq_pullback}.
Hence it suffices to show that $ \cali_{\pi^{-1}(o)}  \otimes 2 \Theta$ is M-regular but not IT(0).
We may assume that $\Theta$ is symmetric.
Since $ h^i(  \cali_{\pi^{-1}(o) +p}  \otimes 2 \Theta) =0$ for any $p \in Y$ and $i \geq 2$,
it suffices to show that 
\[
V^1:=\{ p \in Y \, | \,  h^1(  \cali_{\pi^{-1}(o) +p}  \otimes 2 \Theta) > 0  \}  \subset Y
\]
is not empty and the codimension is greater than $1$.

Since $(X,L)$ is general, $(Y,\Theta)$ is indecomposable, i.e.\ $(Y,\Theta)$ is not a product of smaller dimensional principally polarized abelian varieties.
Hence $|2 \Theta|$ gives a double cover $f:  Y \rightarrow Y/(-1_Y) \subset \P^{2^g-1}$.
Thus $ p \in Y$ is contained in $V^1$ if and only if the restriction map
\[
H^0( Y,2 \Theta)=H^0(\P^{2^g-1},\calo(1)) \rightarrow 2 \Theta |_{ \pi^{-1}(o) +p}
\]
is not surjective
if and only if $f( \pi^{-1}(o) +p) $ is contained in a line in $\P^{2^g-1} $.
Hence we have
\[
V^1=\{ p \in Y \, | \, \text{$f( \pi^{-1}(o) +p) $ is contained in a line}  \}  .
\]
Let $\ep \in \pi^{-1}(o) \simeq \Z/3 \Z$ be a generator.
If $2 p  =\ep$ for $p \in Y$, 
then $ f(p) =f(-p) =f(-\ep +p) $ holds and hence $f( \pi^{-1}(o) +p) = \{ f(p), f(\ep+p), f(-\ep +p) \}$ is contained in a line.
Thus $ \mu_2^{-1} (\ep) $ is contained in $V^1$.
By a similar argument,
we have
\begin{align}\label{eq_V^1}
\mu_2^{-1} (\pi^{-1}(o) ) = \mu_2^{-1} (\{ o_Y, \ep, -\ep \}) \subset V^1.
\end{align}
In particular, $V^1$ is not empty and hence $\cali_o \langle \frac23 l\rangle $ is not IT(0).

By \cite[Theorem (0.5)]{MR769160},
$V^1$ is at most one dimensional.
If $g \geq 3$,
the codimension of $V^1 \subset Y$ is at least $g-1 \geq 2$
and hence $\cali_o \langle \frac23 l\rangle $ is M-regular.

When $g=2$,
we need to see $ \dim V^1 =0$ for the M-regularity of $\cali_o \langle \frac23 l\rangle $. 
Assume $ \dim V^1 \geq 1$.
Then 
$Y$ is the jacobian $J(C)$ of a smooth curve $C$ of genus two and
$\mu_2(V^1) \subset Y=J(C)$ is a theta divisor, i.e.\ 
$\mu_2(V^1)  $ is the image of $C$ by $a \in C \mapsto \calo_C(a) \otimes Q \in J(C)$ for a line bundle $Q$ of degree $-1$ on $C$ 
by \cite[Theorem (0.5)]{MR769160}.
On the other hand,
$\pi^{-1}(o)  = \{ o_Y, \ep, -\ep \}$ is contained in $\mu_2(V^1)$ by \ref{eq_V^1}.
Hence there exist $a,b,c \in C$ such that 
\[
\calo_C(a) \otimes Q =o_Y, \quad  \calo_C(b) \otimes Q  =\ep,  \quad  \calo_C(c) \otimes Q  = -\ep.
\]
Thus $\ep = \calo_C(b-a), -\ep = \calo_C(c-a)$ 
and hence $ \calo_C(b+c) \simeq  \calo_C(2a)$.
Since $b \neq a, c \neq a$ by $\ep \neq o_Y$,
$ \calo_C(2a)$ is basepoint free of degree two on a curve $C$ of genus two.
Thus $ \calo_C(2a)$ is linearly equivalent to the canonical line bundle $\omega_C$.

Furthermore, we also have $ \calo_C(c+a) \simeq  \calo_C(2b) $ by $ \calo_C(b+c) \simeq  \calo_C(2a)$ and $o_Y=3\ep =\calo_C(3b-3a)$.
Thus $ \calo_C(2b) $ is also basepoint free of degree two, and hence $ \calo_C(2b) \simeq \omega_C$.
Then 
\[
 \calo_C(2b) \simeq \omega_C \simeq  \calo_C(2a) \simeq  \calo_C(b+c),
\]
which implies $ b=c$.
This contradicts $\ep = \calo_C(b-a) \neq -\ep = \calo_C(c-a)$.
Hence $  V^1$ cannot be one dimensional and we have the the M-regularity of $\cali_o \langle \frac23 l\rangle $.
\end{ex}

\begin{rem}\label{rem_etc}
(1)
In \cite[Conjecture 6.4]{MR2008719},
it is conjectured that 
if $A$ is ample, $m(A) \geq m$, and $p \geq m$, then $A^n$ satisfies $(N_p) $ for any $n \geq p+3-m$.
The case $m=0$ is nothing but Lazarsfeld's conjecture 
and the case $m=1$ is nothing but \autoref{thm_Pareschi-PopaII}.
If we want to apply \autoref{thm_M-regular_is_enough},
we need to show the M-regularity of $\cali_o \langle \frac{p+3-m}{p+2} a \rangle $.
Since $m(A) \geq m$ and $p \geq m$, it is enough to see the M-regularity of $\cali_o \langle \frac{3}{m(A)+2} a \rangle $.
However, the author does not know this holds or not in general.

\vspace{1mm}
\noindent
(2)
Let $(X,L)$ be a general polarized abelian variety of dimension $g$ and of type $(1,\dots,1,d)$.
In \cite{Ito:2020ab}, the author shows that for an integer $m \geq 1$,
$ \cali_o \langle  \frac{1}{m} l \rangle$ is GV if and only if $ d \geq m^g$
and IT(0) if $ d \geq m^g+ m^{g-1} + \dots +m+1$.

On the other hand,
$L$ has no base divisor if $d > 1=1^g$,
and $L$ is very ample if $d  > 2^g$ by \cite[Remark 3, Corollary 25]{MR1299059},
and $L $ satisfies $(N_1)$ if $ d > 9=3^2$ when $g=2$ by \cite{MR1602020}.
Hence it might be natural to guess that 
$ \cali_o \langle  \frac{1}{m} l \rangle$ is M-regular if $ d > m^g$.
If this is true,
we can recover the above results in \cite{MR1299059}, \cite{MR1602020}
from Theorems \ref{thm_M-regular_is_enough}, \ref{thm_jet-ample}.

We also note that
$ \cali_o \langle  \frac{1}{m} l \rangle$ is not M-regular if $ d \leq  m^g$.
More generally,
$\cali_o \langle x l  \rangle$ is not M-regular if $ x \leq 1/ \sqrt[g]{\chi(l)} $ for any polarized abelian variety $(X,L)$.
In fact, we have an inequality 
\[
h^1_{\cali_o}(yl) \geq h^0_{\calo_X/\cali_o} (yl) - h^0_{\calo_X} (yl) = 1- \chi(l) y^g
\]
for $y \in \Q_{>0}$
by the exact sequence $0 \rightarrow \cali_o \rightarrow \calo_X \rightarrow \calo_X/\cali_o \rightarrow 0$ (see \cite[Section 2.3]{jiang2020cohomological}).
By this inequality, it is easy to see that $ h^1_{\cali_o}((x-t)l) \neq O(t^2)$ for sufficiently small $t \in \Q_{>0}$ if $ x \leq 1/ \sqrt[g]{\chi(l)} $
and hence $\cali_o \langle x l  \rangle$ is not M-regular by \ref{eq_h^i}.
\end{rem}

\section{On Projective normality}\label{sec_proj_normal}


Contrary to Lazarsfeld's conjecture,
the statement of \autoref{thm_Pareschi-PopaII} does not hold for $p=0$,
that is, $A^2$ might not be projectively normal even if $A$ has no base divisor as in the following example.
Hence the statement of \autoref{thm_M-regular_is_enough} does not hold for $p=0$ as well.
Equivalently, the M-regularity of $\cali_o \langle \frac12 l \rangle$ does not imply the projective normality of $L$ in general.

\begin{ex}\label{ex_not_proj_normal}
If 
$A$ is a symmetric ample line bundle on an abelian variety $X$,
$A^2$ is projectively normal if and only if 
the origin $o \in X$ is not contained in the base locus of $ |A \otimes P|$
for any line bundle $P$ on $X$ with $P^2 \simeq \calo_X$ by \cite{MR966402} in $\Char(\K) \neq 2$ and  \cite{MR1433226} in $\Char(\K) \geq 0$.
Thus $A^2$ might not satisfy $(N_0) $ even if $A$ has no base divisor. 
More explicitly,
if $(X,A) $ is a general polarized abelian surface of type $(1,2)$,
then $A $ has no base divisor but
$A^2$ is not projectively normal by \cite[Lemma 6]{MR1106182}.
\end{ex}



In \cite[\S 5]{MR2008719},
the authors give an alternative proof of the above result in \cite{MR966402} in $\Char(\K) \neq 2$ using the theory of M-regularity.
For general $L$, which is not necessarily written as $A^2$, we can show the following lemma:

\begin{lem}\label{lem_proj_normal}
Assume $\Char(\K) \neq 2$ and
let $L$ be an ample line bundle on an abelian variety $X$.
Then the following are equivalent:
\begin{itemize}
\item[(i)] $L$ is projectively normal,
\item[(ii)] $L \otimes {\mu_2}_* (L^{-2})$ is globally generated at the origin $o \in X$,
\item[(iii)] Let $X_2:=\mu_2^{-1}(o)$ be the set of two torsion points in $X$ and 
$f : X \rightarrow \P^N$ be the morphism defined by $|\mu_2^* L \otimes L^{-2}| $.
Then $f(X_2) \subset \P^N$ spans a linear subspace of dimension $\# X_2 -1=4^g-1$.
\end{itemize}

\end{lem}

\begin{proof}
(ii) $\Leftrightarrow $ (iii): 
We note that the line bundle $\mu_2^* L \otimes L^{-2}$ is basepoint free since $\mu_2^* L \otimes L^{-2} \equiv L^2$.
Since $ L \otimes {\mu_2}_* (L^{-2}) ={\mu_2}_* (\mu_2^* L \otimes L^{-2})$ by projection formula,
the natural map
\begin{align}\label{eq_restriction}
H^0(L \otimes {\mu_2}_* (L^{-2})  ) \rightarrow L \otimes {\mu_2}_* (L^{-2}) \otimes \calo_X/\cali_o
\end{align}
can be identified with the natural map
\begin{align}\label{eq_restriction2}
H^0(\mu_2^* L \otimes L^{-2}) \rightarrow \mu_2^* L \otimes L^{-2} \otimes \calo_X/\mu_2^*\cali_o =\mu_2^* L \otimes L^{-2} \otimes \calo_X/\cali_{X_2}.
\end{align}
Since (ii) and (iii) are equivalent to the surjectivity of \ref{eq_restriction} and \ref{eq_restriction2} respectively,
we have the equivalence (ii) $\Leftrightarrow $ (iii).

\vspace{1mm}
\noindent
(i) $\Leftrightarrow $ (ii): 
By \cite[Proposition 2.1]{MR1974682},
$L$ is projectively normal if and only if the natural map $H^0(L) \otimes H^0(L) \rightarrow H^0(L^2)$ is surjective.
We note that the proof of \cite[Proposition 2.1]{MR1974682} works in any characteristic.
Hence $L$ is projectively normal if and only if $L \hat{*} L = L \otimes \phi_l^* \Phi_{\calp}(L^2)$ is globally generated at $o$ by \autoref{prop_Pontrjagin}.

For simplicity, set $\cale=\phi_l^* \Phi_{\calp}(L^2)$ and $ \calf= {\mu_2}_* (L^{-2})$.
Since $\phi_l \circ \mu_2 =\phi_{2l} $, we have
\[
{\mu_2}_* \mu_2^* \cale = {\mu_2}_* \mu_2^* \phi_l^* \Phi_{\calp}(L^2) 
= {\mu_2}_*  \phi_{2l}^* \Phi_{\calp}(L^2) =  {\mu_2}_*  (L^{-2})^{\oplus h^0(L^2)} = \calf ^{\oplus h^0(L^2)},
\]
where the third equality follows from \cite[Proposition 3.11 (1)]{MR607081}.
Since we assume $\Char(\K) \neq 2$, 
$\cale$ is a direct summand of ${\mu_2}_* \mu_2^* \cale =\calf ^{\oplus h^0(L^2)}$.
Hence $L \hat{*} L = L \otimes \cale$ is globally generated at $o$ if so is $L \otimes \calf =L \otimes {\mu_2}_* (L^{-2}) $,
which shows (ii) $\Rightarrow $ (i).

On the other hand, we have
\[
{\phi_l}^* {\phi_l}_* \calf = {\phi_l}^* {\phi_l}_* {\mu_2}_* (L^{-2}) = {\phi_l}^* {\phi_{2l}}_*  (L^{-2})  
=   {\phi_l}^* ( \Phi_{\calp}(L^2))^{\oplus h^0(L^2)} = \cale^{\oplus h^0(L^2)} ,
\]
where the third equality follows from \cite[Proposition 3.11 (2)]{MR607081}.
Since the natural homomorphism $ {\phi_l}^* {\phi_l}_* \calf \rightarrow \calf$ is surjective,
$L \otimes \calf =L \otimes {\mu_2}_* (L^{-2}) $ is globally generated at $o$ if so is $L \hat{*} L = L \otimes \cale$,
which shows (i) $\Rightarrow $ (ii).
\end{proof}

\begin{rem}
(1)
If $L=A^2$ for a symmetric $A$, we have $\mu_2^* A= A^{4}$ and hence
\[
L \otimes {\mu_2}_* (L^{-2}) =  {\mu_2}_* ( \mu_2^* L \otimes L^{-2}) ={\mu_2}_* ( \mu_2^* A^2 \otimes A^{-4}) 
= {\mu_2}_*  {\mu_2}^* A =\bigoplus_{P} A \otimes P
\]
in $\Char(\K) \neq 2$,
where we take the direct sum of all $P$ with $P^2 \simeq O_X$. 
Thus \autoref{lem_proj_normal} recovers Ohbuchi's result when $\Char(\K) \neq 2$.

\vspace{1mm}
\noindent
(2)
By \autoref{thm_JP_Caucci}, \autoref{rem_proj_normal_beta},
$L$ is projectively normal if $\beta(l) <1/2$
and is not projectively normal if $\beta(l) > 1/2$.
Hence the projective normality of $L$ is determined by $\beta(l)$ when $\beta(l) \neq 1/2$. 
It might be interesting to find examples with $\beta(l) =1/2$ to which we can apply \autoref{lem_proj_normal} to show the projective normality,
other than the case $L=A^2$.
We note that we cannot use \autoref{thm_criterion} (1) to show the globally generation of $L \otimes {\mu_2}_* (L^{-2})$ at $o$ when $\beta(l)=1/2$.
In fact,
\[
L \otimes {\mu_2}_* (L^{-2}) \left\langle  -\frac12 l \right\rangle=  {\mu_2}_* ( \mu_2^* L \otimes L^{-2})  \left\langle  -\frac12 l \right\rangle
\]
is not M-regular by \autoref{lem_pushforward}
since $  \mu_2^* L \otimes L^{-2}   \left\langle  -2 l \right\rangle =L^2     \left\langle  -2 l \right\rangle $ is not M-regular.

\vspace{1mm}
\noindent
(3)
Finally,
we summarize relations between projective normality and related notions:
\begin{itemize}
\item[(a)] $\cali_o \langle \frac12 l \rangle$ is IT(0), i.e.\ $\beta(l) < \frac12$,
\item[(b)] $\cali_o \langle \frac12 l \rangle$ is M-regular,
\item[(c)] $\cali_o \langle \frac12 l \rangle$ is GV, i.e.\ $\beta(l) \leq \frac12$,
\item[(d)] $L$ is projectively normal,
\item[(e)] $L$ is very ample.
\end{itemize}
For these five notions,
we have the following relations:

\begin{figure}[htbp]
  \begin{center}
    \includegraphics[scale=0.6]{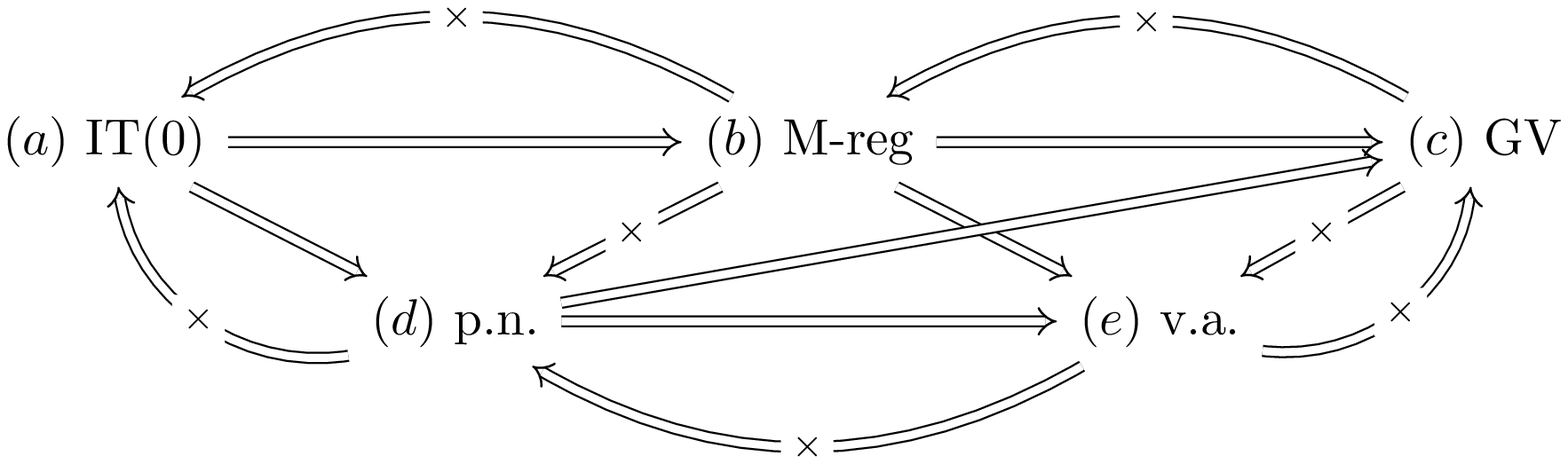}
  \end{center}
\end{figure}

In fact, 
(a) $\Rightarrow$ (b)  $\Rightarrow$ (c) and (d) $\Rightarrow$ (e) hold by definition.
It is easy to see that the converses of these implications do not hold in general by considering suitable $L=A^2$.

(a) $\Rightarrow$ (d), (b) $\Rightarrow$ (e), and (d) $\Rightarrow$ (c) follow from 
\autoref{thm_JP_Caucci}, \autoref{thm_jet-ample}, and \autoref{rem_proj_normal_beta}, respectively

On the other hand,
(b) $\Rightarrow$ (d) and (d) $\Rightarrow$ (a) do not hold in general by \autoref{ex_not_proj_normal} and \autoref{rem_beta=1/2}, respectively.
(c) $\Rightarrow$ (e) does not hold in general since  $ L=A^2$ for a principal polarization $A$ satisfies (c) but is not very ample. 
(e) $\Rightarrow$ (c) also does not hold in general 
since a general polarized abelian $4$-fold $(X,L)$ of type $(1,1,1,15) $ is very ample by \cite[Remark 26]{MR1299059} 
but $\beta(l) \geq 1/\sqrt[4]{\chi(l)} =1/\sqrt[4]{15} > 1/2 $ by \cite[Lemma 3.4]{Ito:2020aa}

The author does not know whether (d) $\Rightarrow$ (b) 
holds or not in general.
\end{rem}

\bibliographystyle{amsalpha}

\begin{thebibliography}{Cau20b}

\bibitem[Ago17]{MR3656291}
Daniele Agostini, \emph{A note on homogeneous ideals of abelian surfaces},
  Bull. Lond. Math. Soc. \textbf{49} (2017), no.~2, 220--225. \MR{3656291}

\bibitem[Bar87]{MR946234}
Wolf Barth, \emph{Abelian surfaces with {$(1,2)$}-polarization}, Algebraic
  geometry, {S}endai, 1985, Adv. Stud. Pure Math., vol.~10, North-Holland,
  Amsterdam, 1987, pp.~41--84. \MR{946234}

\bibitem[BL04]{MR2062673}
Christina Birkenhake and Herbert Lange, \emph{Complex abelian varieties},
  second ed., Grundlehren der Mathematischen Wissenschaften [Fundamental
  Principles of Mathematical Sciences], vol. 302, Springer-Verlag, Berlin,
  2004. \MR{2062673}

\bibitem[Blo19]{Bloss:2019aa}
Patrick Blo{\ss}, \emph{The infinitesimal torelli theorem for hypersurfaces in
  abelian varieties}, arXiv:1911.08311, 2019.

\bibitem[BS97a]{MR1439201}
Th. Bauer and T.~Szemberg, \emph{Higher order embeddings of abelian varieties},
  Math. Z. \textbf{224} (1997), no.~3, 449--455. \MR{1439201}

\bibitem[BS97b]{MR1376538}
\bysame, \emph{Primitive higher order embeddings of abelian surfaces}, Trans.
  Amer. Math. Soc. \textbf{349} (1997), no.~4, 1675--1683. \MR{1376538}

\bibitem[Cau20a]{MR4114062}
Federico Caucci, \emph{The basepoint-freeness threshold and syzygies of abelian
  varieties}, Algebra Number Theory \textbf{14} (2020), no.~4, 947--960.
  \MR{4114062}

\bibitem[Cau20b]{CaucciThesis}
Federico Caucci, \emph{The basepoint-freeness threshold, derived invariants of
  irregular varieties, and stability of syzygy bundles}, Ph.D. thesis, 2020.

\bibitem[Deb06]{MR2233707}
Olivier Debarre, \emph{On coverings of simple abelian varieties}, Bull. Soc.
  Math. France \textbf{134} (2006), no.~2, 253--260. \MR{2233707}

\bibitem[DHS94]{MR1299059}
O.~Debarre, K.~Hulek, and J.~Spandaw, \emph{Very ample linear systems on
  abelian varieties}, Math. Ann. \textbf{300} (1994), no.~2, 181--202.
  \MR{1299059}

\bibitem[GP98]{MR1602020}
Mark Gross and Sorin Popescu, \emph{Equations of {$(1,d)$}-polarized abelian
  surfaces}, Math. Ann. \textbf{310} (1998), no.~2, 333--377. \MR{1602020}

\bibitem[Ito20a]{Ito:2020aa}
Atsushi Ito, \emph{Basepoint-freeness thresholds and higher syzygies on abelian
  threefolds}, arXiv:2008.10272v2, 2020.

\bibitem[Ito20b]{Ito:2020ab}
\bysame, \emph{Higher syzygies on general polarized abelian varieties of type
  $(1,\dots,1,d)$}, arXiv:2011.09687, 2020.

\bibitem[Iye03]{MR1974682}
Jaya~N. Iyer, \emph{Projective normality of abelian varieties}, Trans. Amer.
  Math. Soc. \textbf{355} (2003), no.~8, 3209--3216. \MR{1974682}

\bibitem[Jia20]{jiang2020cohomological}
Zhi Jiang, \emph{Cohomological rank functions and syzygies of abelian
  varieties}, arXiv:2010.10053, 2020.

\bibitem[JP20]{MR4157109}
Zhi Jiang and Giuseppe Pareschi, \emph{Cohomological rank functions on abelian
  varieties}, Ann. Sci. \'{E}c. Norm. Sup\'{e}r. (4) \textbf{53} (2020), no.~4,
  815--846. \MR{4157109}

\bibitem[Kem89]{MR980300}
George~R. Kempf, \emph{Linear systems on abelian varieties}, Amer. J. Math.
  \textbf{111} (1989), no.~1, 65--94. \MR{980300}

\bibitem[Koi76]{MR480543}
Shoji Koizumi, \emph{Theta relations and projective normality of {A}belian
  varieties}, Amer. J. Math. \textbf{98} (1976), no.~4, 865--889. \MR{480543}

\bibitem[MR15]{MR3338009}
David McKinnon and Mike Roth, \emph{Seshadri constants, diophantine
  approximation, and {R}oth's theorem for arbitrary varieties}, Invent. Math.
  \textbf{200} (2015), no.~2, 513--583. \MR{3338009}

\bibitem[Muk81]{MR607081}
Shigeru Mukai, \emph{Duality between {$D(X)$} and {$D(\hat X)$} with its
  application to {P}icard sheaves}, Nagoya Math. J. \textbf{81} (1981),
  153--175. \MR{607081}

\bibitem[Mum66]{MR204427}
D.~Mumford, \emph{On the equations defining abelian varieties.\,{I}}, Invent.
  Math. \textbf{1} (1966), 287--354. \MR{204427}

\bibitem[Mum70]{MR0282975}
David Mumford, \emph{Varieties defined by quadratic equations}, Questions on
  {A}lgebraic {V}arieties ({C}.{I}.{M}.{E}., {III} {C}iclo, {V}arenna, 1969),
  Edizioni Cremonese, Rome, 1970, pp.~29--100. \MR{0282975}

\bibitem[Nak96]{MR1393263}
Michael Nakamaye, \emph{Seshadri constants on abelian varieties}, Amer. J.
  Math. \textbf{118} (1996), no.~3, 621--635. \MR{1393263}

\bibitem[Ohb87]{MR871633}
Akira Ohbuchi, \emph{Some remarks on ample line bundles on abelian varieties},
  Manuscripta Math. \textbf{57} (1987), no.~2, 225--238. \MR{871633}

\bibitem[Ohb88]{MR966402}
\bysame, \emph{A note on the normal generation of ample line bundles on abelian
  varieties}, Proc. Japan Acad. Ser. A Math. Sci. \textbf{64} (1988), no.~4,
  119--120. \MR{966402}

\bibitem[Ohb93]{MR1106182}
\bysame, \emph{A note on the normal generation of ample line bundles on an
  abelian surface}, Proc. Amer. Math. Soc. \textbf{117} (1993), no.~1,
  275--277. \MR{1106182}

\bibitem[Ohb96]{MR1433226}
\bysame, \emph{On the normal generation of ample line bundles on abelian
  varieties defined over some special field}, J. Math. Tokushima Univ.
  \textbf{30} (1996), 1--9. \MR{1433226}

\bibitem[Par00]{MR1758758}
Giuseppe Pareschi, \emph{Syzygies of abelian varieties}, J. Amer. Math. Soc.
  \textbf{13} (2000), no.~3, 651--664. \MR{1758758}

\bibitem[PP03]{MR1949161}
Giuseppe Pareschi and Mihnea Popa, \emph{Regularity on abelian varieties {I}},
  J. Amer. Math. Soc. \textbf{16} (2003), no.~2, 285--302. \MR{1949161}

\bibitem[PP04]{MR2008719}
\bysame, \emph{Regularity on abelian varieties {II}. {B}asic results on linear
  series and defining equations}, J. Algebraic Geom. \textbf{13} (2004), no.~1,
  167--193. \MR{2008719}

\bibitem[PP08]{MR2435838}
\bysame, \emph{{$M$}-regularity and the {F}ourier-{M}ukai transform}, Pure
  Appl. Math. Q. \textbf{4} (2008), no.~3, Special Issue: In honor of Fedor
  Bogomolov. Part 2, 587--611. \MR{2435838}

\bibitem[PP11]{MR2807853}
\bysame, \emph{Regularity on abelian varieties {III}: relationship with generic
  vanishing and applications}, Grassmannians, moduli spaces and vector bundles,
  Clay Math. Proc., vol.~14, Amer. Math. Soc., Providence, RI, 2011,
  pp.~141--167. \MR{2807853}

\bibitem[Rub98]{MR1638159}
Elena Rubei, \emph{Projective normality of abelian varieties with a line bundle
  of type {$(2,\cdots)$}}, Boll. Unione Mat. Ital. Sez. B Artic. Ric. Mat. (8)
  \textbf{1} (1998), no.~2, 361--367. \MR{1638159}

\bibitem[Wel84]{MR769160}
G.~E. Welters, \emph{A criterion for {J}acobi varieties}, Ann. of Math. (2)
  \textbf{120} (1984), no.~3, 497--504. \MR{769160}

\end{thebibliography}
\providecommand{\bysame}{\leavevmode\hbox to3em{\hrulefill}\thinspace}
\providecommand{\MR}{\relax\ifhmode\unskip\space\fi MR }
\providecommand{\MRhref}[2]{%
  \href{http://www.ams.org/mathscinet-getitem?mr=#1}{#2}
}
\providecommand{\href}[2]{#2}

\end{document}